\documentclass{amsart}
\usepackage{amssymb}

\usepackage{graphicx}

\def\cal{\mathcal}

\input{tcilatex}
\input tcilatex

\newcommand{\EE}{{\mathbb E}}

\newcommand{\PP}{{\mathbb P}}
\newcommand{\RR}{{\mathbb R}}

\newcommand{\F}{{\cal F}}

\newcommand{\T}{{\cal T}}
\newcommand{\1}{{\bf 1}}
\newcommand{\vP}{\overleftarrow{P}}

\newcommand{\vX}{\overleftarrow{X}}
\newcommand{\hP}{\widehat P}

\newcommand{\Id}{\hbox{Id}}
\newcommand{\ac}{\alpha}
\newcommand{\bc}{\beta}
\newcommand{\cc}{\gamma}
\newcommand{\hX}{\widehat X}

\newcommand{\tX}{\widetilde X}
\newcommand{\tP}{\widetilde P}
\newcommand{\pa}{\partial}
\newcommand{\hI}{\widehat I}
\newcommand{\tI}{\widetilde I}
\newcommand{\hT}{\widehat{\cal T}}
\newcommand{\tT}{\widetilde{\cal T}}
\newcommand{\phT}{\widehat{T}}
\newcommand{\ptT}{\widetilde{T}}
\newcommand{\tpa}{\widetilde{\partial}}
\newcommand{\ovP}{\overline P}
\newcommand{\ovX}{\overline X}

\begin{document}
\title[duality and intertwining]
{Duality and Intertwining for discrete Markov kernels: a relation and examples}
\author{Thierry Huillet$^{1}$, Servet Martinez$^{2}$}
\address{$^{1}$Laboratoire de Physique Th\'{e}orique et Mod\'{e}lisation \\
CNRS-UMR 8089 et Universit\'{e} de Cergy-Pontoise, 2 Avenue Adolphe Chauvin,
95302, Cergy-Pontoise, FRANCE\\
$^{2}$ Departamento de Ingenier\'{i}a Matem\'{a}tica \\
Centro Modelamiento Matem\'{a}tico\\ UMI 2807, UCHILE-CNRS \\
Casilla 170-3 Correo 3, Santiago, CHILE.\\
E-mail: Thierry.Huillet@u-cergy.fr and smartine@dim.uchile.cl}
\maketitle

\begin{abstract}
We work out some relations between duality and intertwining 
in the context of discrete Markov chains, fixing up the background of  
previous relations first established for birth and death chains and their 
Siegmund duals. In view of the results, the monotone properties resulting 
from the Siegmund dual of birth and death chains are revisited in some detail, 
with emphasis on the non neutral Moran model. We also introduce an 
ultrametric type dual extending the Siegmund kernel.  
Finally we discuss the sharp dual, following closely the 
Diaconis-Fill study.
\end{abstract}

\smallskip

\noindent{\bf{Running title}}: Duality and Intertwining.\newline

\smallskip

\noindent{\bf{Keywords}}: {\it Duality; Intertwining; Siegmund dual; 
generalized ultrametric matrices; birth and death chains; Moran 
models; sharp strong stationary time; sharp dual.}

\smallskip

\noindent{\bf{MSC 2000 Mathematics Subject Classification}}: 60 J 10, 
60 J 70, 60 J 80.  

\bigskip

\section{Introduction}
Our work is devoted to the study of duality and intertwining relations between 
Markov chain kernels. Even if these concepts can be established only as 
relations between matrices, as we define them in the next section, our 
study is on its probabilistic consequences. For this purpose we need that
the matrices are non negative and substochastic to be able to define a 
dual Markov chain. The fact that the intertwining kernel is stochastic
allows a rich probabilistic interpretation that has been given in 
\cite{cpy}, \cite{DF}, \cite{F} and \cite{F2}.

\bigskip

A main problem is the existence of a duality 
relationship between substochastic 
kernels. Indeed, once this fact is established, then several relations
can be deduced when the starting chain is irreducible and 
positive recurrent. This is the statement of one of our main result, 
which is Theorem \ref{theo1}. The hypotheses of this theorem rely on 
a duality relationship between kernels. 

\bigskip

In the following sections, we find additional examples where these 
duality relations between substochastic kernels can be established:
for the well-known Siegmund kernel the hypotheses of Theorem \ref{theo1}
are verified for monotone chains, see Corollary \ref{coro21}; 
for a generalized ultrametric potential kernel some conditions  
for the existence of the dual are given in Proposition \ref{prop7};  
for birth and death (BD) chains the properties derived from 
monotonicity are summarized in Corollary \ref{coro31}; and  
in Proposition \ref{propo6} we show that 
the non-neutral Moran model is monotone when its bias 
mechanism is nondecreasing. 

\bigskip

For birth and death chains, we 
revisit the properties relating non negative spectrum 
and monotonicity (see Proposition \ref{propo5}) and
for the Moran model we identify some cases with 
non negative spectra and also when stronger properties are 
satisfied.

\bigskip

The section \ref{sec5} follows closely the ideas on sharp stationary 
times 
and duals developed in  \cite{AD}, \cite{DF} and \cite{F}. 
In Proposition \ref{prop9} we show a sharpness result alluded
to in Remark $2.39$ of \cite{DF} and in Theorem 2.1 in \cite{F}.
One of its corollaries is Proposition \ref{prop9'}
where the condition for sharpness is written in terms of 
the dual function. This applies to the intertwining 
of a monotone chain under the Siegmund dual, in this case both chains
can start from the state $0$. In the BD case we also study 
some quantitative aspects of the absorption time.

\bigskip

We point out that even if duality and intertwining can be set 
for Markov chains acting on general state spaces 
and/or with continuous time,
we restrict ourselves to the discrete time and space in order to be 
able to present quickly our main results and avoid to introduce 
additional overburden notations.

\section{ Duality and Intertwining }
\label{s1}

\subsection{Notation}
Let $I$ be a countable set. 
By $\F(I)$ we denote the set of real functions, and by
$\F_b(I)$ and $\F_+(I)$ we denote respectively its bounded
and positive elements. Since $I$ is countable the set $\F(I)$ 
is identified with the set of vectors $\RR^I$. 
Let $\pa$ be a point that does not
belong to $I$, and denote $I^\pa:=I\cup \{\pa\}$. Every
$f\in \F(I)$ is extended canonically to a function $f^\pa$ that 
satisfies $f^\pa(\pa)=0$. 

\bigskip

If $A$ is any set we denote by $\1_A$ or  $\1(A)$ its characteristic 
function. We denote by $\1$ the unit function defined on $I$ (or in 
other sets $\hI$ and $\tI$ that we introduce further).

\bigskip

A non negative matrix $P=(P(x,y): x,y\in I)$ is called a kernel on 
$I$. (Sometimes we will emphasize the non negativity by saying
a non negative kernel.) It obviously acts on the set $\F_+(I)$. 
A substochastic kernel is such that $P \1\le \1$, it is stochastic
when the equality $P \1 = \1$ holds, and strictly substochastic if
it is substochastic and there exists some $x\in I$ such that
$P \1(x)<1$. When $P$ is substochastic, it obviously acts on  $\F_b(I)$.

\medskip

The kernel $P$ is irreducible when for any pair $x,y\in I$
there exists $n>0$ such that $P^{(n)}(x,y)>0$. 

\medskip

A point $x_0\in I$ is an absorbing point of the kernel $P$
when $P(x_0,y)=\delta_{y,x_0}$ for all $y\in I$. 

\medskip

When $P$ is a substochastic kernel there exists a uniquely defined 
(in distribution) Markov chain $X=(X_n: n< \T^X)$
taking values in the countable set $I$, with
lifetime $\T^X$ and with transition kernel $P$. We have
the equality $P=P_X$
where $P_X$ is the kernel acting on 
the set of functions $\F_b(I)$ (or $\F_+(I)$) by 
$$
P_X f(x)=\EE(f(X_1) \cdot \1(\T^X>1))\,,\; x\in I\,.
$$
$P$ generates the semigroup $(P^n: n\ge 1)$, each 
matrix $P^n$ acting on $\F_b(I)$ or 
$\F_+(I)$, and it verifies  
$$
P^n f(x)=\EE(f(X_n) \cdot \1(\T^X>n))\,, \;\; x\in I\,, \, n\ge 1\,. 
$$

The lifetime $\T^X$ is such that

\begin{itemize}
\item If $P$ is stochastic then
$\T^X=+\infty$ $\PP_x-$a.e. for all $x\in I$;

\item If $P$ is strictly substochastic then there exists some $x\in I$ 
such that $\PP_x(\T^X<+\infty)>0$. When $P$ is irreducible
strictly substochastic then for all $x\in I$ it holds 
$\PP_x(\T^X<+\infty)=1$.
\end{itemize}

The kernels will be denoted by $P$, $\hP$, $\tP$, they will be 
defined on the countable sets $I$, $\hI$, $\tI$ 
respectively. When these kernels are substochastic the associated 
Markov chains will be respectively denoted by $X$, $\hX$, 
$\tX$, and the lifetimes of these chains will be respectively $\T$, 
$\hT$, $\tT$. 

\subsection{ Strictly substochastic kernel}
If $P$ is strictly substochastic we can add a new state
$\pa$ to $I$, and $X$ is extended 
to the Markov chain $X^\pa=(X^\pa_t: t\ge 0)$ by
$$
X^\pa_t=X_t\,,\; t<\T\,; \;\;\; X^\pa_t=\pa \,,\; t\ge \T\,,
$$
so $\pa$ is an absorbing state of the new chain. The transition kernel 
$P_{X^\pa}$ of $X^\pa$ is stochastic and it is given by
$$
P_{X^\pa} g (x) = \EE_x(g(X^\pa_1) \cdot \1(T^{X^\pa}_{\pa}>1))
+g(\pa) \PP_x(T^{X^\pa}_{\pa}\le 1)\,,
$$
for all $g\in \F_b(I^\pa)$ or $g\in \F_+(I^\pa)$.
Then,
\begin{equation}
\label{equatn1}
\left[g(\pa)=0 \right]\; \Rightarrow \; 
\left[\left(P^n_{X^\pa}\, g\right)\big|_{I}
= P^n \left( g\big|_{I}\right) \,, \;\; \forall\,  n\ge 1\right].
\end{equation}
Therefore, since the canonical extension of $f\in \F(I)$ to
$f^\pa\in \F(I^{\pa})$ satisfies $f^\pa (\pa)=0$,
the right hand side of (\ref{equatn1}) is verified
for $g=f^\pa$.

\bigskip

We recall that $h\in \F_b(I)$ (or $h\in \F_+(I)$) is a
$P-$harmonic
function if $P h=h$, or equivalently if it verifies
$$
\EE_x(h(X_n)\cdot \1(\T > n))=h(x) \,,\;\; \forall x\in I\,,
\forall n\ge 1\,.
$$
We have that its extension $h^\pa\in \F_b(I^{\pa})$ (or $h^\pa\in
\F_+(I^{\pa})$) such that $h^\pa(\pa)=0$ is a
$P_{X^\pa}-$harmonic function.

\bigskip

Let us denote by
$$
T^X_J=\inf\{n\ge 0: X_n \in J\}
$$
the hitting time of $J\subseteq I$ of the chain $X$, 
where as usual we put $+\infty=\inf \; \emptyset$. 
When $J=\{a\}$ is a singleton we put $T^X_a$ 
instead of $T^X_{\{a\}}$. Observe that with this notation we have
$$
\T^X=T^{X^{\pa}}_{\pa}.
$$

To simplify the notation, for the Markov chains $X$, $\hX$, $\tX$,
the hitting times are
denoted respectively by $T_J=T^X_J$, ${\phT}_J=T^{\hX}_J$,
${\ptT}_J=T^{\tX}_J$
(when $J$ is a subset of $I$, $\hI$, $\tI$, respectively).

\bigskip

Let us recall the structure of a non irreducible substochastic 
kernel $P$. In this case, up to permutation, 
we can partition 
$$
I=\bigcup_{l=1}^\ell I_l
$$
in such a way that (see \cite{HJ}, Section 8.3): 
$$
P_{I_l\times I_l} \hbox{ is irreducible } \; \forall
l\in \{1,\cdots, \ell\}\,,
$$
and 
$$
\forall \; x\in I_l\,, \, y\in I_{l'}\,: \;\, P(x,y)>0 \;
\Rightarrow \; l\le l'\,.
$$
If $P$ is stochastic then the last of these submatrices 
$P_{I_\ell\times I_\ell}$ is stochastic,
that is $P_{I_\ell\times I_\ell} \1_{I_\ell}=\1_{I_\ell}$
and there could be other stochastic submatrices. If $P$ is strictly 
substochastic then none or some of these submatrices
$P_{I_l\times I_l}$, $l=1,\cdots, \ell$, could be stochastic. We put
$$
St(P)=\{I_l: P_{I_l\times I_l}
\hbox{ is stochastic}, \, l\in \{1, \cdots ,\ell\}\}\,.
$$
Then, when $P$ is stochastic $St(P)\neq \emptyset$, and if
$P$ is strictly substochastic then $St(P)$ could be 
empty or not. When  $St(P)\neq \emptyset$ then it could 
contain a unique class or not, and also by a simple permutation 
we can always assume that it contains $I_{\ell}$ (this permutation
is not needed when $P$ is stochastic).

\subsection{ Definitions }
We recall the duality and the intertwining relations. As usual 
$M'$ denotes the transposed of matrix $M$, that is $M'(x,y)=M(y,x)$
for all $x,y\in I$.  

\medskip

\begin{definition}
\label{def1}
Let $P$ and  $\hP$ be two kernels defined on the countable sets 
$I$ and  $\hI$, and let $H=(H(x,y): x\in I, y\in \hI)$ be a 
non negative matrix. Then $\hP$ is said to be a $H-$dual of $P$ if
it is verifies
\begin{equation}
\label{equatn2}
H \hP'=P H\,.
\end{equation}
We call $H$ a dual function between $(P,\hP)$. $\Box$
\end{definition}

Note that for a kernel $P$ the $H-$dual $\hP$ exists when 
(\ref{equatn2}) holds and $\hP\ge 0$.

\medskip

When $|I|=|\hI|$ is finite and $H$ is nonsingular we get that 
$$
\hP'=H^{-1}P H\,,
$$
and so $\hP'$ and $P$ are similar matrices and have the 
same spectrum.

\bigskip

Duality is a symmetric notion between kernels, because if $\hP$ is a 
$H-$dual of $P$, then $P$ is a $H'-$dual of $\hP$.

\medskip

We will assume that the non negative dual matrix $H$ is nontrivial,
in the sense that no row and no column vanishes completely.
On the other hand note that if $H$ is a dual function between $(P, \hP)$
then for all $c>0$, $cH$ is also a dual function between these
matrices. Then, when it is necessary, we can always multiply all the
coefficients of $H$ by a strictly positive constant.

\bigskip

This notion of duality (\ref{equatn2}) coincides with the one between 
Markov processes
that can be found in references \cite{Lig}, \cite{MM} and
\cite{DF} among others. Indeed, let $P$ and $\hP$ be substochastic 
and let $X$ and $\hX$ be Markov chains 
with kernels $P$ and $\hP$ respectively. Then, if $\hP$ is 
a $H-$dual of $P$, we have that $\hX$ is a $H-$dual of $X$,
which means that
\begin{equation}
\label{equatn3}
\forall \, x\in I, \, y\in  \hI,\, \forall \, n\ge 0\, : 
\;\;\;
\EE_x(H(X_{n}, y))=\EE_y(H(x, \hX_{n}))\,,
\end{equation}
where we have extended $H$ 
to $(I\cup \{\pa\})\times (\hI\cup \{\pa\})$ by putting
$H(x, \pa)=H(\pa,y)=H(\pa,\pa)=0$, for all $x\in I$,
$y\in \hI$.

\bigskip

Let us now introduce intertwining.

\bigskip

\begin{definition}
\label{def2}
Let $P$ and  $\tP$ be two kernels defined on the countable sets
$I$ and  $\tI$ and let $\Lambda=(\Lambda(y,x): y\in \tI, x\in I)$ be 
a stochastic matrix. We say that $\tP$ is a $\Lambda-$intertwining 
of $P$, if it verifies
\begin{equation*}
\tP \Lambda= \Lambda P \,.
\end{equation*}
$\Lambda$ is called a link between $(P,\tP)$. $\Box$
\end{definition}

When $|I|=|\tI|$ is finite and $\Lambda$ is nonsingular we get 
$$ 
\tP= \Lambda P \Lambda^{-1}\,.
$$
and so $P$ and $\tP$ are similar and have the same spectrum.

\bigskip

Let $P$ and $\tP$ be substochastic and denote by 
$X$ and $\tX$ the associated Markov chains, if $\tP$ is a $\Lambda-$ 
intertwining of $P$ we say that $\tX$ is a $\Lambda-$intertwining of 
$X$. Obviously the
intertwining is not a symmetric relation because $\Lambda'$
is not necessarily stochastic. But when $\Lambda$ is doubly stochastic  
we have that $\tP$ is a $\Lambda-$intertwining of $P$ implies that
$P$ is a $\Lambda'-$intertwining of $\tP$.

\bigskip

The stochastic intertwining
between Markov chains has been deeply studied in   
\cite{cpy}, \cite{DF}, \cite{F} and \cite{F2}.

\section{ Relations between Duality and Intertwining }
\label{s2}

Let us introduce additional notation:
\begin{itemize}

\item By $\mathbf{e}_{a}$ we denote a column vector with $0$
entries except for its $a-$th entry which is $1$;

\item When $P$ is an irreducible positive recurrent stochastic 
kernel, we denote by $\pi=(\pi(x): x\in I)$ its stationary distribution
and we write it as a column vector. So $\pi'P=\pi'$, 
where $\pi'$ is the row vector transposed of $\pi$.
\end{itemize}

Now we give a result on intertwining that will be often used.

\begin{proposition}
\label{prop0}
Let $P$ be an irreducible positive recurrent
stochastic kernel and $\pi$ be its stationary
distribution. Assume $\tP$     
is a kernel that is a $\Lambda-$ intertwining of $P$,  
$\tP \Lambda= \Lambda P$. If $\widetilde a$ is an absorbing state in 
$\tP$ then,
\begin{equation}
\label{equatn6''}
\pi'=\mathbf{e}_{\widetilde a}'\Lambda.
\end{equation}
\end{proposition}

\begin{proof}
Since the chain $P$ is positive recurrent
with stationary distribution $\pi$ and $\Lambda$ is stochastic we get
$\lim\limits_{k\to\infty}\frac{1}{k}\sum_{n=0}^{k-1}
(\Lambda {P}^n)(x,y)=\pi(y)$, in particular
\begin{equation}
\label{equatn61}
\lim\limits_{k\to\infty}\frac{1}{k} \sum_{n=0}^{k-1}
(\Lambda{P}^n)(\widetilde a,y)=\pi(y)\,.
\end{equation}
On the other hand from the assumption we get 
$\tP^n(\widetilde a,y)=\delta_{y,\widetilde a}$
and then
\begin{equation}
\label{equatn62}
({\tP}^n \Lambda)(\widetilde a,y)=\Lambda(\widetilde a,y) 
\;\; \forall n\ge 0\,, y\in I\,.
\end{equation}
From (\ref{equatn5}) we have
${\tP}^n \Lambda= \Lambda {P}^n$ for all $n\ge 1$, and so from
(\ref{equatn61}) and (\ref{equatn62}) we deduce
$\Lambda(\widetilde a,y)=\pi(y)$ which is equivalent to
$(\mathbf{e}_{\widetilde a}'\Lambda)(y)=\pi(y)$.
Then (\ref{equatn6''}) is shown.
\end{proof}    

\bigskip

For a vector $\rho\in \RR^I$ we will denote by
$D_\rho$ the diagonal matrix with terms $(D_\rho)(x,x)=\rho(x)$, 
$x\in I$.

\bigskip

Let $P$ be an irreducible positive recurrent stochastic kernel with
stationary distribution $\pi$. 
By irreducibility we have $\pi>0$.
Denote by $\vP$ the transition kernel of the
reversed chain of $X$, so $\vP(x,y)={\pi(x)}^{-1} P(y,x)\pi(y)$ or 
equivalently
\begin{equation}
\label{equatn4}
\vP'=D_\pi P D_\pi^{-1}.
\end{equation}
We have that $\vP$ is in duality with $P$ via $H=D_\pi^{-1}$.
Note that $\vP$ is also irreducible and positive recurrent with 
stationary distribution $\pi$ and that $P'=D_\pi \vP D_\pi^{-1}$,
so we can exchange the roles of $P$ and $\vP$. 
In the reversible case $\vP=P$, the relation (\ref{equatn4})  
expresses a self duality.

\bigskip

Let us give one of our main results that can be viewed as the 
generalization of Theorem 5.5 in \cite{DF} devoted to birth and 
death chains.

\bigskip

\begin{theorem}
\label{theo1}
Let $P$ be an irreducible positive recurrent
stochastic kernel and let $\pi$ be its stationary
distribution. Assume $\hP$ is a (non negative) kernel and 
that it is a $H-$dual of $P$, $H{\hP}'=P H$, where $H$ 
is nontrivial. Then 

\medskip

\noindent $(i)$ ${\hP}H'D_\pi=H'D_\pi \vP$.

\medskip       

\noindent $(ii)$ The vector $\varphi:= H' \pi$ is strictly positive
and it verifies
$$
{\hP} \varphi=\varphi\,.
$$ 

\noindent $(iii)$ $\tP=D_\varphi^{-1} {\hP} D_\varphi$ is a 
stochastic kernel and it is a
$\Lambda-$intertwining of $\vP$, so $\Lambda$ is a stochastic link $\Lambda$, 
more precisely
\begin{equation}
\label{equatn5}
\tP \Lambda= \Lambda \vP \, \hbox{ with } 
\Lambda:=D_\varphi^{-1} H' D_\pi 
\hbox{ and they verify } \tP \1 = \1 =\Lambda \1 \,.
\end{equation}
Moreover we have the duality relation
$$
K \tP'= P K\, \hbox{ with } K:=H D_\varphi^{-1}\,.
$$

\noindent $(iv)$ Let $I$ and $\hI$ be finite and $\hP$ be substochastic. 
Then:

\medskip       

\noindent $(iv1)$ When $\hP$ is stochastic and irreducible then 
$\varphi=c \1$ for some $c>0$, and $\tP=\hP$.

\medskip       

\noindent $(iv2)$ 
If $\hP$ is strictly substochastic then it is not irreducible.

\noindent $(iv3)$
If  $\hP$ is non irreducible then $St(\hP)\neq \emptyset$ and there 
exist some constants $c_l > 0$ for $\hI_l\in St(\hP)$ such that
\begin{equation}
\label{equatn6}
\varphi(x)= \sum\limits_{{\hI}_l\in St(\hP)} c_l 
\PP_x(\lim\limits_{n\to \infty}\hX_n\in {\hI}_l)=
\sum\limits_{{\hI_l}\in St(\hP)} c_l \PP_x({\phT}_{\hI_l}<\hT)\,.
\end{equation}

\noindent $(iv4)$ If $\hP$ has a unique stochastic class ${\hI}_\ell$,
then,  
\begin{equation}
\label{equatn7'}
\frac{\varphi(x)}{\varphi(y)}=\PP_x( {\phT}_{\hI_\ell}<\hT) \; \; 
\hbox{ for any } y\in {\hI}_\ell\,,
\end{equation}
and the intertwining Markov chain $\tX$ is given by the Doob transform
\begin{equation}
\label{equatn7}
\PP_x(\tX_1=y_1,\cdots, \tX_k=y_k)=
\PP_x(\hX_1=y_1,\cdots, \hX_k=y_k \, | \, {\phT}_{{\hI}_\ell}<\hT)\,.
\end{equation}

\noindent $(v)$ If $\widehat a$ is an absorbing state in $\hP$ then
$\widehat a$ is an absorbing state in $\tP$ and (\ref{equatn6''})
$\pi'=\mathbf{e}_{\widehat a}'\Lambda$ holds. Moreover the sets of 
absorbing points in $\hP$ and $\tP$ coincide.      

\bigskip

\noindent $(vi)$ If $|I|=|\hI|$ is finite and $H$ is nonsingular then: 
$\hP =H'D_\pi \vP D_\pi^{-1}{H'}^{-1}$
and $\tP=\Lambda \vP \Lambda^{-1}$. 
Hence $\hP$, $\vP$, $\tP$, 
are similar matrices and $P$, $\hP$, $\tP$ have the same spectrum. 
\end{theorem}

\begin{proof}
From $H {\hP} '= P H$, we find
$$
{\hP} H'= H' D_\pi \vP D_\pi^{-1} \,.
$$
By multiplying to the right by $D_\pi$ we get $(i)$. The part $(vi)$
follows directly in the finite nonsingular case.

\bigskip
 
Since $D_\pi \1=\pi$ we get that ${\hP} H' \pi=H' D_\pi \vP \1$.
Since $\vP$ is stochastic we get ${\hP} H' \pi=H' D_\pi \1=H' \pi$. Let
$\varphi= H' \pi$. Since $\pi>0$ and at each row
of $H$ there exists a strictly positive element, then $\varphi>0$.
Then $(ii)$ holds. Now define,
$$
\tP=D_\varphi^{-1} \hP D_\varphi\,.
$$
By using $(i)$ we get,
$$
\tP D_\varphi^{-1} H' D_\pi = D_\varphi^{-1} \hP H'  D_\pi= 
D_\varphi^{-1} H' D_\pi \vP
$$
Then the relation $\tP \Lambda= \Lambda \vP$ holds in $(iii)$, 
moreover
\begin{eqnarray}
\nonumber
\tP\1&=&D_\varphi^{-1} \hP D_\varphi \1=D_\varphi^{-1} 
\hP \varphi=D_\rho^{-1} \varphi=\1 
\,,\\
\nonumber
\Lambda \1 &=& D_\varphi^{-1} H' D_\pi \1= D_\varphi^{-1} 
H' \pi=D_\varphi^{-1} \varphi=\1\,.
\end{eqnarray}
Then $\tP$ and $\Lambda$ are Markov kernels. Finally
from the equality 
$$
H D_\varphi^{-1} \tP'  D_\varphi = H \hP'=P H \,,
$$
the relation $K \tP'= P K$ is straightforward. Hence $(iii)$
is verified.

\bigskip

Now assume $I$ is finite.
If $\hP$ is an irreducible strictly substochastic kernel then necessary
its spectral radius is strictly smaller that $1$, which contradicts
the equality $\hP \varphi=\varphi$, because $\varphi>0$. In the case 
$\hP$ is stochastic 
and irreducible, the equation  $\hP \varphi=\varphi$, $\varphi>0$, 
implies $\varphi=c \1$ for some constant $c>0$.  So $(iv1)$ and $(iv2)$ 
follow. 

\bigskip

Now assume that the matrix $\hP$ is substochastic and non irreducible. 
Let $\hI=\bigcup_{l=1}^\ell {\hI}_l$ be the partition
in irreducible components $\hP_{{\hI}_l\times {\hI}_l}$ such that
$x\in {\hI}_l$, $y\in {\hI}_{l'}$ and $\hP(x,y)>0$ implies $l'\ge l$.
The last submatrix $\hP_{{\hI}_\ell\times {\hI}_\ell}$ verifies,
$$
\hP_{{\hI}_{\ell}\times {\hI}_{\ell}}\, \varphi \big|_{{\hI}_{\ell}}=
\varphi \big|_{{\hI}_{\ell}}\,.
$$
Then $\hP_{{\hI}_\ell\times {\hI}_\ell}$ 
is an irreducible substochastic matrix whose Perron-Frobenius 
eigenvalue is $1$, so we deduce that $\hP_{{\hI}_\ell\times {\hI}_\ell}$ 
is stochastic and $\varphi \big|_{{\hI}_\ell}=c_\ell \1_{{\hI}_\ell}$
for some constant $c_\ell>0$, so $St(\hP)\neq \emptyset$.
Then, if $({\hI}_l\in St(\hP))$ are the irreducible 
stochastic classes the same argument implies that 
$\varphi \big|_{{\hI}_l}=c_l \1_{{\hI}_l}$ for some quantity
$c_l>0$ and this happens for all ${\hI}_l\in St(\hP)$.

\bigskip

Let $\hX=(\hX_t: t<\hT)$ be the Markov chain with kernel $\hP$.
It is known that all the trajectories that are not killed are 
attracted by $\bigcup\limits_{{\hI}_l\in St(\hP)}{\hI}_l$, that is
\begin{equation*}
\PP_x(\lim\limits_{n\to \infty}\hX_n\in 
\bigcup\limits_{{\hI}_l\in St(\hP)}{\hI}_l \; | \; \hT=\infty)=1.
\end{equation*}
On the other hand
the equality $\hP \varphi=\varphi$ expresses that $\varphi$ is 
an harmonic function for the chain $\hX$. Hence, for all $n\ge 0$ 
it is verified,
\begin{eqnarray*}
\varphi(x)&=&\EE_x(\varphi(\hX_n), \hT>n)\\
&=&\sum\limits_{{\hI}_l\in St(\hP)}\EE_x(\varphi(\hX_n), \hT>n, 
{\phT}_{{\hI}_l}<\hT)\\
\nonumber
&{}&+ 
\EE_x(\varphi(\hX_n), \hT>n, \hT<\min\{{\phT}_{{\hI}_l}: {\hI}_l\in 
St(\hP)\}) 
\,. 
\end{eqnarray*}
Then, by taking $n\to \infty$ in above expression and since 
$\lim\limits_{n\to \infty}\PP_x(\min\{{\phT}_{{\hI}_l}: {\hI}_l\in
St(\hP)>\hT>n)=0$, we get the relation (\ref{equatn6}),
$$
\varphi(x)=\sum\limits_{{\hI}_l\in St(\hP)} 
c_l\, \PP_x({\phT}_{{\hI}_l}<\hT)\,.
$$

Let us prove part $(iv4)$. Since there is a unique stochastic class 
the equality (\ref{equatn7'}) follows straightforwardly.
Then the transition probabilities of $\tP$ are
given by the Doob $h-$transform
$$
\tP(x,y)= \PP_x( {\phT}_{{\hI}_\ell}<\hT)^{-1} \hP(x,y)
\PP_y({\phT}_{{\hI}_\ell}<\hT)=
\PP_x(\hX_1=y \, | \,  {\phT}_{{\hI}_\ell}<\hT) \,, \; \;
\forall \, x,y\in \hI\,.
$$
The Markov property gives the formula for every cylinder.

\bigskip

Finally, let us show part $(v)$. 
Since the chain $\vP$ is positive recurrent
with stationary distribution $\pi$ it suffices to show that
${\widehat a}$ is an absorbing state for $\tP$.
This follows straightforwardly from the equality
$\tP=D_\varphi^{-1} {\hP} D_\varphi$, indeed it implies
$\tP(\widehat a,y)=\hP(\widehat a,y)
\frac{\varphi(y)}{\varphi(\widehat a)}=\delta_{y,\widehat a}$.
Also this proves the equality of the set of absorbing points 
for both kernels $\hP$ and $\tP$.
\end{proof}

\bigskip

\begin{remark}
\label{rem1}
We can exchange the roles of $P$ and $\vP$ in
the irreducible and positive recurrent case. Thus, in the hypothesis
of the Theorem we can take $\vP$ instead of $P$, so $\hP$ is $H-$dual of 
$\vP$, $H {\hP}'=\vP H$, and in all the statements of the Theorem we 
must change $P$ by $\vP$. $\Box$
\end{remark}

\begin{remark}
\label{rem4}
A probabilistic explanation of how appears $\varphi:= H' \pi>0$
can be done when $\hP$ is substochastic and $H$ is bounded.
In this case the dual relation $H{\hP}'=P H$ is expressed by
the expression (\ref{equatn3}),
$$
\forall \, x\in I, \, y\in  \hI,\, \forall \, n\ge 0\, :
\;\;\;
\EE_x(H(X_{n}, y))=\EE_y(H(x, \hX_{n}))\,.
$$
Since by hypothesis $X$ is an irreducible and positive recurrent
Markov chain then $\varphi$ appears as
the following limit on the left hand side,
$$
\lim\limits_{k\to\infty}\frac{1}{k}\sum\limits_{n\le k}
\EE_x(H(X_{n},y))=\sum\limits_{u\in I} \pi(u) H(u,y)=\varphi(y)\,.
$$
$\Box$
\end{remark}

\begin{remark}
\label{rem3}
We have
\begin{equation}
\label{equatn6'}
\Lambda(x,y)=\frac{1}{\varphi(x)}H(y,x) \pi(y)\,,
\end{equation}
in particular $\Lambda(x,y)=0$ if and only if $H(y,x)=0$.$\; \Box$
\end{remark}

\begin{remark}
\label{rem3'}
The formulas in Theorem \ref{theo1} state that $\Lambda$, $\hP$ and $\tP$ 
are invariant when $H$ is multiplied by a strictly positive constant.  
Then, we can fit $c>0$ and take $cH$ in order to have 
$\varphi(x)=1$ for all $x\in {\hI}_l$, 
or equivalently $c_l=1$, for some fixed stochastic class 
${\hI}_l\in St(\hP)$.
\end{remark}

\begin{remark}
\label{rem5'}
When the starting equality between stochastic kernels is the 
intertwining relation ${\tP} \Lambda= \Lambda {\vP}$, then
we have the duality relation $H \hP'=PH$ with 
$H=D_\pi^{-1}\Lambda'$ and $\hP=\tP$. In this case $\varphi=\1$. 
\end{remark}

We note the equality $\hI=\tI$ of the sets where the kernels $\hP$ and 
$\tP$ are defined in Theorem \ref{theo1}. On the other hand we
recall that in the finite case the positive recurrence property on $P$ 
follows from irreducibility. 

\bigskip

\begin{proposition}
\label{propo1}
Assume $H$ is nonsingular and has a constant column that is strictly 
positive, that is
$$
\exists {\widehat a}\in \hI \,: \;\;  H\mathbf{e}_{\widehat a}=c\, \1\,
\hbox{ for some } \, c>0\,.
$$
Then:

\medskip

$(i)$ ${\widehat a }$ is an absorbing state for $\hP$ 
(so $\{\widehat a\}$ is a stochastic class). 

\medskip

\noindent $(ii)$ Under the hypotheses of Theorem \ref{theo1},
$\pi'=\mathbf{e}_{\widehat a}'\Lambda$ holds and if
$\hP$ is strictly substochastic and $\{\widehat a\}$ is the 
unique stochastic class then 
$\PP_y( {\phT}_{\widehat a}<\hT)=\varphi(y)/\varphi(\widehat a)$ and the 
relation (\ref{equatn7}) is satisfied.
\end{proposition}

\begin{proof}
$(i)$ From $H\mathbf{e}_{\widehat a}=c\, \1$ we get,
\begin{eqnarray*}
\mathbf{e}'_{\widehat a}\hP &=&\mathbf{e}'_{\widehat a}
H'D_\pi \vP D_\pi^{-1} H^{-1\prime } \!=
\left( H \mathbf{e}_{\widehat a}\right)^{\prime }D_\pi
\vP D_\pi^{-1} H^{-1\prime }\\
&=& c\,  \pi' \vP D_\pi^{-1} H^{-1\prime }\!=
\left( H^{-1} c\, \1 \right)^{\prime }=\mathbf{e}'_{\widehat a}.
\end{eqnarray*}
Then 
$$
\hP(\widehat a,y)=(\mathbf{e}_{\widehat a}^{\prime }\hP)(y)
=\mathbf{e}'_{\widehat a}(y)=\delta_{\widehat a,y}.
$$
So $\widehat a$ is an absorbing state for $\hP$.

\bigskip

\noindent $(ii)$ From Theorem \ref{theo1} $(v)$, $\widehat a$ 
is an absorbing state of $\tP$ and $\pi'=\mathbf{e}_{\widehat a}'\Lambda$. 
The rest of part $(ii)$ follows straightforwardly. 
\end{proof}

When $P$ does not satisfy positive recurrence let us
only consider the following special case.

\begin{proposition}
\label{propo2}
Let $x_0\in I$ be an absorbing point of the kernel $P$ and let 
$\hP$ be a substochastic kernel that is a $H-$dual of $P$: $H \hP'=PH$. 
Then $h(y):=H(x_0, y)$, $y\in \hI$, is a non negative $\hP-$harmonic
function. When $H$ is bounded and $\hP$ is a stochastic recurrent 
kernel, the $x_0-$row $H(x_0,\cdot)$ is constant.
\end{proposition}

\begin{proof}
It suffices to show that the function $h$ is  $\hP-$harmonic.
Since $P(x_0,z)=\delta_{z,x_0}$ $\; \forall \, z\in I$, we get
$(P H) (x_0,y)= H (x_0,y)$. Therefore, if
$\hP$ verifies the duality equality (\ref{equatn2}) we get,
$(H \hP') (x_0,y) = H(x_0,y)=h(y)$. Then
$$
(\hP h)(y)=\sum\limits_{z\in \hI} \hP(y,z) H(x_0,z)=
\sum\limits_{z\in \hI} H(x_0,z) \hP'(z,y)=(H \hP') (x_0,y)=h(y)\,,
$$
and the result is shown.
\end{proof}

\section{ Classes of Dual matrices }

We consider the finite set case. We assume 
$I=\hI=\tI=\{0,\cdots, N\}$, so the kernels
are non negative $I\times I$ matrices and when they are substochastic 
the associated Markov chains take values in $I$.

\bigskip

We will study some classes of non negative matrices $H$ for which there 
exist substochastic kernels $P$ and $\hP$ in duality 
relation (\ref{equatn2}). So, in these cases we would be
able to apply the results established in Theorem \ref{theo1},
and Proposition \ref{propo1}. 

\subsection{ The potential case }
\label{sub34}
Let us see what happens with a quite general class of matrices, the finite
potential kernels.
Let $R$ be a strictly substochastic kernel with no stochastic 
classes (this is the case if $R$ is also irreducible). Then it has
a well defined finite potential,
$$
H=(\Id-R)^{-1}=\sum\limits_{n\ge 0} {R}^n\ge 0\,.
$$
So $H^{-1}=\Id-R$. (In particular no column nor row of $H$ vanishes). 

\bigskip

Let $P$ be a substochastic kernel. Define
$$
\hP'=H^{-1} P H=(\Id-R) P (\Id-R)^{-1} \,.
$$

\begin{proposition}
\label{propo3}
Assume that also the transposed matrix $R'$ is substochastic. Then,
$\hP {\1}\ge 0$ and
there exists a stochastic kernel $P$ for which it is verified $\hP\ge 0$. 
Indeed, the constant stochastic kernel $P=\frac{1}{N+1} {\1} {\1}'$ 
fulfills the property.
\end{proposition}

\begin{proof}
Since $R'$ is substochastic we have $(\Id-R') \1\ge 0$. Then
$$
\1'\hP'= \1' (\Id-R)P (\Id-R)^{-1}\ge 0\,.
$$

Now, since $(\Id-R)^{-1}\ge 0$ and $\hP'= (P-RP)(\Id-R)^{-1}$,
we get that once the relation $RP\le P$ is verified then $\hP'\ge 0$.
Since $R$ is substochastic the matrix
$P=\frac{1}{N+1} {\1} {\1}'$ makes the job. 
\end{proof}

\subsection{ Siegmund kernel}
A well-known case of a kernel $H$ arising as a potential of a strict 
substochastic kernel $R$ as above, is the Siegmund kernel. Let 
$R(x,y)=\1(x+1=y)$ so
it is a strictly substochastic (because the $N-$th row vanishes) and it 
has no stochastic classes. Its transposed matrix 
$R'(x,y)= \1(x=y+1)$ is also substochastic. 

\bigskip

By direct computation we get that $H_S=(\Id-R)^{-1}$ verifies 
$$
H_S(x,y)=\1(x\le y)
$$ 
so it is the Siegmund kernel. We have
$H_S^{-1}=\Id-R$, then $H_S^{-1}(x,y)= \1(x=y)-\1(x+1=y)$.

\bigskip

This case has been
studied in detail, for instance see \cite{DF} Section 5.
Let us summarize some well-known observations. We have
$(H_S \hP')(x,y)=\sum\limits_{z\ge x}\hP(y,z)$ and
$(PH_S)(x,y)=\sum\limits_{z\le y}P(x,z)$.
Then, the equation $H_S \hP'= PH_S$ gives
\begin{equation}
\label{equatn8}
\hP(y,x)=\sum\limits_{z\ge x} \hP(y,z)-\sum\limits_{z> x}\hP(y,z)=
\sum\limits_{z\le y}(P(x,z)- P(x+1,z))
\end{equation}
In particular $\hP\ge 0$ requires the condition,
\begin{equation}
\label{equatn8'}
\forall y\in I\,: \; 
\sum\limits_{z\le y}P(x,z) \hbox{ decreases with } \, x\in I. 
\end{equation}
In this case $P$ is called monotone.

\medskip

Also, from $\hP(N,x)=\sum\limits_{z\le N} (P(x,z)-P(x+1, z))$
we deduce that
\begin{equation*}
P \hbox{ stochastic }\; \Rightarrow \; \hP(N,x)=\delta_{x,N},
\end{equation*}
so $N$ is an absorbing state of $\hP$.
Also from (\ref{equatn8}) we get that
\begin{equation}
\label{equatn9}
\hP(N-1,N)=\sum\limits_{z\le N-1}P(N,z)=1-P(N,N)\,.
\end{equation}
We also observe that
\begin{equation*}
\sum\limits_{x\le N} \hP(y,x)=\sum\limits_{z\le y}P(0,z)\,,
\end{equation*}
in particular 
\begin{equation}
\label{equatn10}
\hP \1\le \1  \, \hbox{ so } \, \hP \hbox{ is substochastic },
\end{equation}
and also $\sum\limits_{x\le N} \hP(0,x)=P(0,0)$. Then,
\begin{eqnarray}
\label{equatn11'}
P(0,0)=1 &{}& \Rightarrow \; \hP \hbox{ is stochastic } \,;\\
\label{equatn11}
P(0,0)<1 &{}& \Rightarrow \;
\sum\limits_{x\le N} \hP(0,x)<1\, \hbox{ and } \hP 
\hbox{ looses mass through } 0.
\end{eqnarray}
This last case occurs for any irreducible stochastic kernel $P$
with $N\ge 1$. Indeed, in this case $P(0,0)=1$ cannot happen because
it contradicts irreducibility. 

\medskip

Also we get,
\begin{equation*}
P(0,0)+P(0,1)=1 \; \Rightarrow \;  \hP \hbox{ does not 
loose mass through } \{1,\cdots, N\}\,.
\end{equation*}

When the finite matrix $P$ is irreducible we can apply 
Theorem \ref{theo1} and in this case
\begin{equation}
\label{equatn12}
\varphi(x)=(H'_S\pi)(x)=\sum\limits_{y\in I} \1(y\le x) \pi(y)=
\sum\limits_{y\le x}\pi(y)=: \pi^c(x)\,,
\end{equation}
is the cumulative distribution of $\pi$. We have that $\pi^c$ is not
constant because $\pi>0$.

\bigskip

Let us show that 
\begin{equation}
\label{eqn01}
N \, \hbox{ is the unique absorbing state of } \, \hP \,.
\end{equation}
Indeed, from (\ref{equatn8}) the unique absorption implies that 
$x<N$ verifies $\hP(x,y)=\delta_{y,x}$ if and only if 
$\sum\limits_{z\le x} P(x,z)=1$ 
and $\sum\limits_{z\le x} P(x+1,z)=0$. Therefore, also from (\ref{equatn8}),
we get 
$$
\sum\limits_{z\le x} P(y,z)=1 \;\; \forall \, y\le x \hbox{ and }
\sum\limits_{z\le x} P(y,z)=0  \;\; \forall \, y> x\,,
$$ 
which contradicts the irreducibility of $P$.

\bigskip

We obtain the following result. In it we assume $N\ge 1$ to
avoid the trivial case when $N=0$ and $P(0,0)=1$.

\begin{corollary}
\label{coro21}
Let $H$ be the Siegmund kernel, $P$ be a monotone 
finite irreducible stochastic 
kernel with stationary distribution $\pi$. Let $H_S{\hP}'=P H_S$ with
$\hP\ge 0$. Then: 

\medskip

\noindent $(i)$ $\hP$ is a strictly substochastic kernel that looses 
mass through $0$, and parts $(iv2)$ and  
$(iv3)$ of Theorem \ref{theo1} hold. 

\medskip

\noindent  $(ii)$ $\varphi=\pi^c$ and the stochastic intertwining 
kernel $\Lambda$ verifies
\begin{equation}
\label{eqn02}
\Lambda(x,y)=\1(x\ge y)\frac{\pi(y)}{\pi^c(x)}\,.
\end{equation}
and the intertwining matrix $\tP$ of $\vP$ is given by
\begin{equation*}
\tP(x,y)=\hP(x,y)\frac{\pi^c(y)}{\pi^c(x)}\,\;\; x,y\in I\,.
\end{equation*}

\noindent  $(iii)$  $N$ is the unique absorbing state for $\hP$ and      
Theorem \ref{theo1} parts $(iv4)$ and $(v)$ are verified with 
${\hI}_\ell=\{N\}$ and ${\widehat a}=N$. In particular 
$\pi'=\mathbf{e}_{N}'\Lambda$
holds.

\medskip

\noindent $(iv)$ The following relation holds:
\begin{equation}
\label{equatn15}
\Lambda \mathbf{e}_N=\pi(N) \mathbf{e}_N 
\end{equation}
\end{corollary}

\begin{proof}
The first three parts are direct consequence of Theorem \ref{theo1}, relations 
(\ref{equatn10}), (\ref{equatn11}), (\ref{equatn12}), (\ref{eqn01}).  
Finally, (\ref{equatn15}) is a direct computation from (\ref{eqn02}) and  
$\varphi(N)=1$. 
\end{proof}

\subsection{ Duality for finite state space birth and death chains }
Recall $I=\hI=\tI=\{0,\cdots, N\}$.
Let $X=(X_n: n\ge 0)$ be a discrete birth and death (BD) chain with 
transition Markov kernel $P=(P(x,y): x,y\in I)$. Then 
$P(x,y)=0$ if $|x-y|>1$ and
$$
P(x,x+1)=p_x, \; P(x,x-1)=q_x\, \; P(x,x)=r_x\,, \;\; x\in I \,,
$$
with 
\begin{equation*}
q_x+r_x+p_x=1 \;\; \forall x\in I \,\; 
\hbox{  and boundary conditions } q_0=p_N=0\,.
\end{equation*}
We always take
\begin{equation*}
q_{x},\, p_{x}>0, x\in \{1,..,N-1\}\,.
\end{equation*}
We will assume the irreducible case, which in this case is equivalent to 
the condition
\begin{equation}
\label{equatn13}
p_{0}>0\,, \;\; q_{N}>0\,. 
\end{equation}
(A unique exception will be done in Subsection \ref{subsec43} where 
we will explicitly assume that (\ref{equatn13}) is not satisfied.)
The stationary distribution $\pi=\left(\pi(x): x\in I\right)$ verifies 
$\pi(y)=\pi(0)\prod\limits_{z<y}
\frac{p_{z}}{q_{z+1}}>0$, $y\in \{1,..,N\}$, where $\pi(0)$ 
fulfills $\sum\limits_{y\in I} \pi(y)=1$. 

\bigskip

The matrix $P$ is self-adjoint with the inner product given by $\pi$, 
that is it verifies 
$\pi(x)P(x,y)=\pi(y)P(y,x)$ for all $x,y$. So, 
$P=\overleftarrow{P}$ where
$\vP= D_{\pi}^{-1}P'D_{\pi}$ is the transition matrix of the time
reversed process and $P$ has real eigenvalues.

\bigskip

The unique constraint in (\ref{equatn8}) to get that
$\hP\ge 0$ is satisfied, is for
$y=x$, that is we need that the condition $\hP(x,x)\ge 0$ is
verified and it reads
\begin{equation}
\label{equatn14}
\forall x\in \{0,\cdots, N-1\}\, : \;\; p_x+q_{x+1}\le 1 \,.
\end{equation}
This is the equivalent of (\ref{equatn8'}) for BD chains.
So, when (\ref{equatn14}) is satisfied we say that $P$ is monotone.
In this case the Siegmund dual $\hP$ exists and it is a BD kernel with
$$
\hP(x,x-1)=p_x \,, \; 
\hP(x,x)=1-(p_x+q_{x+1}) \,, \; 
\hP(x,x+1)=q_{x+1} \,.
$$
The drift  of $X$ at $x$
is $f(x):=p_{x}-q_{x}$, and the drift of
$\widehat{X}$ at $x$ is
$\widehat{f}( x) =\widehat{p}_{x}-\widehat{q}(x)=-f(x+1)
+( p_{x+1}-p_{x})$, so
$-f( x+1) -r_{x+1}\leq \widehat{f}(x) \leq -f(x) +r_{x}$.

\bigskip

Note that $\hP(0,0)=1-(p_0+q_1)$, $\hP(0,1)=q_1$, then the Markov chain $\hX$  
looses mass through the state $0$ if and only if $p_0=0$ or 
equivalently 
$r_0<1$. On the other hand
$\hP(N,N)=1-(p_N+q_{N+1})=1$ (because $p_N=q_{N+1}=0$), so 
$N$ is an absorbing 
state. When (\ref{equatn14}) holds we say that $P$ is a monotone kernel.
From this analysis, Corollary \ref{coro21}, and the reversibility
relation $P=\vP$ we can state the following result.

\begin{corollary}
\label{coro31}
Let $H$ be the Siegmund kernel and $P$ be a finite irreducible 
stochastic BD chain with monotone kernel $P$ and whose 
parameters are $p_x, q_x$. Let $\pi$ be the 
stationary distribution of $P$. Then, the dual matrix
$\hP$ defined by ${\hP}'=H_S^{-1} P H_S$ is a strictly substochastic 
kernel that looses mass through the state $0$. Moreover:

\medskip

\noindent $(i)$  $\varphi=\pi^c$ and parts $(iv2)$ and $(iv3)$ 
of Theorem \ref{theo1} hold.

\medskip

\noindent  $(ii)$ $N$ is an absorbing state
of $\hP$ and $\{N\}$ is the unique stochastic class of $\hP$,  
all the other states in $I$ are transient, and
Theorem \ref{theo1} $(iv4)$ is verified with ${\hI}_\ell=\{N\}$.

\medskip

\noindent  $(iii)$ Let $\Lambda$ be the stochastic kernel given by 
(\ref{eqn02}). Then, the $\Lambda-$intertwining matrix $\tP$ 
of $P$ is given by
$$
\tP(x,x-1)=p_x \frac{\pi^c(x-1)}{\pi^c(x)}\,, \;
\tP(x,x)=1-(p_x+q_{x+1})\,, \;
\tP(x,x+1)=q_{x+1}\frac{\pi^c(x+1)}{\pi^c(x)}\,.
$$
\end{corollary}

\subsection{Absorbing points for the BD kernels}
\label{subsec43}
Let us modify the BD kernel $P$ by taking $0$ as an 
absorbing state. That is, instead of the irreducibility 
conditions (\ref{equatn13}) we take $p_{0}=0$ and no restriction
on $q_{N}$, it could be $0$ or $>0$.
Assume $P$ is monotone, so (\ref{equatn14}) 
holds. Then the BD kernel $\hP$ is stochastic, see (\ref{equatn11'}). 
In this case $N$ is the unique absorbing state for $\hP$.

\bigskip

Let us describe what happens by exploiting the special form of the 
Siegmund dual. By evaluating (\ref{equatn3}) at $y=0$ we get
\begin{equation}
\label{equatn16}
\PP_x(X_n\le 0))=\PP_0(x\le \hX_n)\,,
\end{equation}
and by evaluating (\ref{equatn3}) at $x=N$ we obtain
\begin{equation}
\label{equatn17}
\PP_N(X_n>y)=\PP_y(\hX_n<N)\,.
\end{equation}

Now there are two cases:

\noindent $(i)$ If $q_N>0$ then $0$ is the unique absorbing state 
for $P$. By (\ref{equatn9}) we get that $\hP(N-1,N)=q_N>0$, so $N$ 
is an absorbing state that attracts all the trajectories of the chain,
$\PP_x(\lim\limits_{n\to \infty} X_n=N)=1$ for $x\in I$.

\bigskip

\noindent $(ii)$ If $q_N=0$, then $0$ and $N$ are absorbing 
states for $P$. By using (\ref{equatn9}) we get that $\hP(N-1,N)=q_N=0$,
so $N$, besides being an absorbing state for $\hP$ is an 
isolated state for $\hP$ (that is $\hP(y,N)=0$ for all $y<N$). 
Therefore it does not attract any of the  trajectories starting 
from a state different from $N$. Hence,
the equation (\ref{equatn17}) is simply the equality $1=1$ when $y<N$. 

\bigskip

Let us summarize which is the picture for  $(ii)$: $P$ has $0$ and 
$N$ as absorbing states that attract all the trajectories of its 
associated Markov 
chain $X$, $\hP$ is stochastic, $N$ is a $\hP-$absorbing isolated 
state, and $\hP\big|_{I\setminus \{N\}\times I\setminus \{N\}}$ is 
stochastic and irreducible. Let 
${\widehat{\pi}}_*=({\widehat{\pi}}_*(z): z\in I\setminus \{N\})$ 
be the stationary distribution of  the submatrix 
$\hP\big|_{I\setminus \{N\}\times I\setminus \{N\}}$. 

\bigskip

Let $\phi(x)=\PP_{x}\left( \lim\limits_{n\to \infty}X_{n}=0\right)$
be the absorption probability at $0$ of the chain $X$ starting from 
$x$. We have the following result.

\begin{proposition}
\label{propo4}
If $p_{0}=0$ and $q_N=0$ then $0$ and $N$ are absorbing states 
for $P$ and 
$\hP\big|_{I\setminus \{N\}\times I\setminus \{N\}}$ is stochastic
and has $0$ as an absorbing point.
Let $\phi(x)=\PP_{x}\left( \lim\limits_{n\to \infty}X_{n}=0\right)$,
and ${\widehat{\pi}}_*=({\widehat{\pi}}_*(z): z\in I\setminus \{N\})$
be the stationary distribution of 
$\hP\big|_{I\setminus \{N\}\times I\setminus \{N\}}$.
Then
$$
\phi(x)=1-\frac{\eta(x) }{\eta(N) }=1-{{\widehat{\pi}}_*}^c(x+1)
$$
where ${{\widehat{\pi}}_*}^c$ is the cumulative distribution
of ${\widehat{\pi}}_*$ and 
$\eta(x):=\sum_{y=0}^{x-1}\prod_{z=1}^{y}\frac{q_{z}}{p_{z}}$
is the scale function of $P$.
\end{proposition}

\begin{proof}
The first equality follows from the fact that $\eta$ is a martingale 
and $\eta(0)=0$. For the second relation
we take $x<N$ and let $n\to \infty$ in the 
formula (\ref{equatn16}), which gives
$$
\phi(x)=\sum\limits_{z\ge x}
{\widehat{\pi}_*}(z)=1-{{\widehat{\pi}}_*}^c(x+1)\,.
$$
\end{proof}

\subsection { The spectral characterization.}
Let us give a sufficient spectral property for the monotonicity of 
the kernel $P$ for an irreducible BD chain taking values on 
$I=\{0,\cdots, N\}$.
Consider the polynomials 
$\left( \frak{q}_{y}(t) : y\in I\right)$ with
$t\in [-1,1]$, determined by: $\frak{q}_{0}(t) =1$ for all $t$
and the recurrence:
\begin{eqnarray*}
t\frak{q}_{0}(t) &=&
p_{0}\frak{q}_{1}(t) + r_{0}\frak{q}_{0}(t) \,,\\
t\frak{q}_{y}(t)&=&p_{y}\frak{q}_{y+1}(t) +
r_{y}\frak{q}_{y}(t) +
q_{y}\frak{q}_{y-1}(t) \,,\;\, y\in \{ 1,\cdots, N-1\} \,.
\end{eqnarray*}
It holds $\frak{q}_{y}(1) =1$ for all $y\geq 0$ and
the polynomial $\frak{q}_{y}(t) $ is of degree $y$ in $t$.

\bigskip

Let $Z:=\left\{t_{k}: R_{N+1}(t_{k}) =0\right\}$ be
the zeros of the polynomial
$R_{N+1}(t) = t\frak{q}_{N}(t) -r_{N}\frak{q}_{N}(t)
-q_{N}\frak{q}_{N-1}(t)$, which is 
of degree $N+1$. The set $Z$ constitutes the
spectrum of $P$ (see \cite{Karlin}, p. 78).
All the zeros are simple and we order them
by $1=t_{0}>t_{1}>..>t_{N}\geq -1$. 
The quantity $1-t_{1}$ is the
spectral gap. The spectral probability measure on $[-1,1]$ is
$\mu( dt) :=\sum_{k=0}^{N}\mu _{k}\delta _{t_{k}}$, with respect to 
which $(\frak{q}_{y}(t) : y\geq 1)$ are orthogonal.
It is known that $\mu_{0}=\pi _{0}$.

\bigskip

Let $N=2N_{0}$ be even. Assume that the BD chain is given 
by $r_x=0$ for all $x\in I$, and that it is  reflected at the 
boundaries $\{ 0,N\}$, so $p_0=q_N=1$. 
In this case the spectral measure
is symmetric on $[-1,1]$, in particular $t_{N_{0}}=0$
and $t_{2N_{0}}=-1$. When $N=2N_{0}+1$ is odd, the spectral measure 
is again symmetric, but $\{ 0\} $ is no longer an eigenvalue 
and $t_{N_{0}}>0$.

\bigskip

A spectral sufficient condition for the monotone property 
(\ref{equatn14}) is given below in part $(i)$. 
This result can be found in Lemma 2.4 of \cite{F}, and here we give 
a different proof. On the other hand  note that
when $r_{x}\geq 1/2\; \, \forall \, x\in I$ then obviously
the monotone condition (\ref{equatn14}) is satisfied.
In part $(ii)$ we reinforce this implication.

\begin{proposition}
\label{propo5}
\noindent $(i)$ If a BD chain is spectrally non negative, then 
it is monotone.

\bigskip

\noindent $(ii)$ If $r_{x}\geq 1/2$ for all $x\in I$, then
the BD chain is spectrally positive. 

\end{proposition}

\begin{proof}
Let us show $(i)$. For a BD chain $X$ whose transition
matrix $P$ is spectrally non negative, there exists a BD chain 
$Y$ taking values on $\{ 0,..,2N\}$ reflected 
at the boundary, started at an even integer and such that
$X \overset{d}{\sim} \left( Y_{2n}/2: n\ge 0\right)$.
This follows simply from adapting \cite{W}, Th. $2.1$ to the
finite case. As noted just before, the spectral measure
of $Y$ is symmetric on $[-1,1] $ and by
passing to $X$ the spectrum is being folded: If
$\sum_{k=0}^{2N}\mu _{k}\delta_{t_{k}}$ is the
symmetric spectral measure of $Y$ with $t_{N}=0$
then $2\sum_{k=0}^{N}\mu _{k}\delta _{t_{k}^{2}}$ is the spectral
measure of $X$. Let $\alpha_{y}$ and $\beta_{y}$ be the up and down
probabilities that $Y_{m}\rightarrow Y_{m+1}=$ $Y_{m}\pm 1$ 
given that $Y_{m}$ is in state $y$ different from the endpoints.
We have $\alpha_{y}+\beta _{y}=1$, and then:
\begin{equation*}
q_{x}=\beta _{2x}\beta _{2x-1},\text{ }r_{x}=\beta _{2x}\alpha
_{2x-1}+\alpha _{2x}\beta _{2x+1},\text{ }p_{x}=\alpha _{2x}\alpha
_{2x+1}.
\end{equation*}
This, together with $p_{0}=\alpha _{1}$ and $q_{N}=\beta _{2N-1}$
allows to determine recursively
the transition matrix of $Y$ from the one of $X$. From these
facts we deduce that our hypothesis implies
\begin{equation*}
p_{x}+q_{x+1}=\alpha_{2x}\alpha _{2x+1}+\beta_{2x+2}
\beta_{2x+1}<\alpha_{2x}\alpha _{2x+1}+\beta _{2x+1}<1\,,
\end{equation*}
then the chain $X$ is monotone. 

\bigskip

For the proof of $(ii)$ first note that 
$P=D_\pi^{-\frac{1}{2}} Q D_\pi^{-\frac{1}{2}}$,
where $Q$ is a symmetric matrix given by
$Q(x,y)=0$ when $|x-y|>1$ and
$$
Q(x,x+1)=\sqrt{p_x q_{x+1}}=Q(x+1,x), \;
Q(x,x)=r_x\,, \,\; x\in  I\,.
$$
Now  consider the superdiagonal
matrix $S$ such that $S(x,y)=0$ if $y\notin \{x,x+1\}$ and
$$
S(x,x)=\sqrt{p_x}\,, \; x\in I; \;\;\;
S(x,x+1)=\sqrt{q_{x+1}}\,, \; x\in I \,, \, x\neq N\,.
$$
Then $S'S$ is a tridiagonal symmetric matrix, with $S'S(x,y)=0$ if
$|x-y|>1$ and
$$
S'S(x,x)=p_x+q_x\,, \; \; S'S(x,x+1)=\sqrt{p_xq_{x+1}}=S'S(x+1,x)\,,\;
x\in I\,.
$$
Consider the diagonal matrix
$D_r$ with $r=(r_{0},\cdots,r_{N})$. We have
$Q=2D_r-I+S^{\prime }S$,
so $Q$ is the sum of a diagonal matrix and a symmetric positive
definite matrix.
We conclude that, if the holding probabilities $r_{x}\geq 1/2$,
for all $x\in I$, then for all $z\in \RR^{N+1}\setminus \{0\}$,
\begin{equation*}
z' Q z=\sum_{x=0}^{N}\left(
2r_{x}-1\right)
\left| z_{x}\right| ^{2}+\left| S z \right| ^{2}>0
\end{equation*}
and so $Q$ and $P$ are positive definite.
\end{proof}

We emphasize that in part $(ii)$ we show that 
$r_{x}\geq 1/2 \; \forall x\in I$
implies that the spectrum is positive, and that this is a 
stronger property than monotonicity in view of $(i)$.
On the other hand the condition $r_{x}\geq 1/2$ for all 
$x\in I$ is sufficient to get a positive spectrum but, 
as it is easy to see, it is not necessary.

\bigskip

\emph{Example:} An example showing that non negative
spectrum is not necessary for the monotone property (\ref{equatn14})
is the BD chain given $p_{x}=p$, $q_{x}=q$, $x=1,..,N-1$ and
boundary conditions $r_{0}=q$, $p_0=p$, $q_N=q$, $r_{N}=p$,
where $p\in (0,1)$ and $q=1-p$.
Then the monotone property holds but the spectrum fail to
be non negative. Indeed, from \cite{Feller}, p. $438$ it follows that
$t_{k}=2\sqrt{pq}\cos (\frac{k\pi }{N+1})$, $k=1,.,N-1$, $t_{0}=1$,
$t_{N}=-1$.$\Box$

\subsection{ The Moran model }
Let us introduce the $2$-allele Moran model with bias
mechanism $p$. Let
\begin{equation*}
p:[0,1] \rightarrow [ 0,1] \hbox{ be continuous with }\, 
0\leq p(0) \text{ and } p(1) \leq 1\,.
\end{equation*}
Denote $q(u):=1-p(u)$.
The Moran model is a BD Markov chain $X$ characterized by
the quadratic transition probabilities $p_x, \, r_x, \,q_x$,
$x\in I=\{0,..,N\}$,
\begin{equation*}
q_{x}=\frac{x}{N}q\left( \frac{x}{N}\right) ,\; r_{x}=
\frac{x}{N} p\left(\frac{x}{N}\right) +
\left( 1-\frac{x}{N}\right) q\left( \frac{x}{N}\right) ,
\; p_{x}=\left( 1-\frac{x}{N}\right) p\left(\frac{x}{N}\right) .
\end{equation*}
Assuming $p_{0}=p\left( 0\right) >0$ and $q_{N}=1-p\left( 1\right)
>0$, with
$y\in \left\{ 1,..,N\right\} $, the BD chain is irreducible
with invariant distribution
\begin{equation*}
\frac{\pi(y)}{\pi(0)}=\prod_{x=1}^{y}\frac{p_{x-1}}{q_{x}}=
\binom{N}{y}\prod_{x=1}^{y}\frac{p\left(\frac{x-1}{N}\right)}
{q(\frac{x}{N})}=
\frac{p(0) \binom{N}{y}}{q\left( \frac{y}{N}\right) }
\prod_{x=1}^{y-1}\frac{p\left( \frac{x}{N}\right) }{q(\frac{x}{N})},
\end{equation*}
where $\pi(0)$ is the normalizing constant.

\bigskip

If $X$ is a Moran model defined by some
bias $p$, then $\overline{X}_{n}:=N-X_{n}$
is also a Moran model with bias
$\overline{p}(u) =1-p(1-u)$, and so with parameters
$$
\overline{q}_{x}=p_{N-x}=\frac{x}{N}\overline{q}\left(
\frac{x}{N}\right),
\;\;
\overline{p}_{x}=q_{N-x}=\left( 1-\frac{x}{N}\right) \overline{p}
(\frac{x}{N})
$$
where $\overline{q}\left( u\right) :=1-\overline{p}(u)$.
The spectra of $\overline{P}$ and $P$ are the same.

\bigskip

\begin{proposition}
\label{propo6}
Assume that in the Moran model the bias $p$ is nondecreasing. Then the
BD chain is monotone, that is
condition (\ref{equatn14}) $p_{x}+q_{x+1}\leq 1$ is fulfilled (and so
the Siegmund dual exists).
\end{proposition}

\begin{proof}
First, since $p_N=q_{N+1}=0$ we have nothing to verify for
$x=N$. Let us see what happens with $x=0$. We need
to guarantee $1-p_{0}-q_{1}\geq 0$, but this is true because
$p(1/N) \geq p( 0) \geq Np( 0) -(N-1)$.

\medskip

Let us consider the case $x\in \{1,\cdots, N-1\}$. We have
the following relations, where
in the first inequality we use that $p$ is nondecreasing,
\begin{eqnarray}
\nonumber
p_{x}+q_{x+1}&=& p\left(\frac{x}{N}\right) -
\frac{x}{N} p\left(\frac{x}{N}\right) +
\left(\frac{x+1}{N}\right)-\left( \frac{x+1}{N}\right)
p\left(\frac{x+1}{N}\right)\\
\nonumber
&\leq & p\left(\frac{x}{N}\right)
\left(1-\frac{x}{N}-\left(\frac{x+1}{N}\right)\right)
+\left(\frac{x+1}{N}\right)\\
\label{equatn19}
&=&\frac{1}{N}\left( p\left(\frac{x}{N}\right)
\left( (N-1-2x) +(x+1) \right)\right)\le 1\,.
\end{eqnarray}
Now, the last inequality $\le 1$ in (\ref{equatn19}) is fulfilled because:
\bigskip

\noindent If $x=\frac{N-1}{2}$ it reduces to
$\frac{x+1}{N}\le 1$;

\bigskip

\noindent If $x<\frac{N-1}{2}$ it reduces to
$p\left(\frac{x}{N}\right)\le \frac{N-x-1}{N-1-2x}$,
and this is satisfied because the right hand side of this
expression is $>1$;

\bigskip

\noindent If $\frac{N-1}{2}<x\le N-1$ it is verified because
$N-1-x\ge 0$ and $N-1-2x< 0$.
\end{proof}

\bigskip

\textbf{Moran model with mutations.} A basic bias example is the
mutation mechanism
\begin{equation}
p(u) =\left( 1-a_{2}\right) u+a_{1}\left( 1-u\right),
\label{equatn20}
\end{equation}
where $\left( a_{1},a_{2}\right) $ are mutation probabilities in
$(0,1] .$ The drift is $p\left( u\right) -u$.
When $a_1+a_2\neq 1$, the invariant probability measure satisfies
$\pi(x)=\binom{N}{x}\frac{\left( \alpha
\right)_{x}\left(\beta \right)_{N-x}}{\left( \alpha +\beta 
\right)_{N}}$,
$x\in I$, where $(\alpha) _{x}:=\Gamma
\left(\alpha +x\right) /\Gamma \left( \alpha \right)$.

\medskip

When $p$ is non-decreasing, we
have $a_{1}+a_{2}\leq 1$. In $\overline{p}(u) $ 
the roles of $a_{1}$ and $a_{2}$ are exchanged.

\medskip

The case $a_{1}=a_{2}=1$, that is  $p(u) =1-u$, 
corresponds to the heat-exchange
Bernoulli-Laplace model \cite{Feller}. Here,
$\pi(x)=\binom{N}{x}\binom{N}{N-x}/\binom{2N}{N}$.
If $a_{1}=a_{2}=1/2$ then
$p(u) =1/2$ which is amenable (through a
suitable time substitution) to the Ehrenfest urn model provided $N$ is
even.

\medskip

One-way mutations, $(a_{1},a_{2}) =
(a_{1},0)$ or $(0,a_{2})$ lead to the choice $p(1)=1$ or 
$p(0)=0$ respectively, corresponding to the case in which 
$N$ or $0$  is an absorbing state respectively.

\bigskip

Except for some exceptional special cases, the spectral measure
associated to the Moran model is not known. Let us supply some
of these special cases.

\medskip

\textbf{Spectral representation of the Moran model with mutations.}
Assume $a_{1}+a_{2}\neq 1,$ \cite{KMG}. Here the eigenvalues are
\begin{equation}
t_{k}=1-\frac{k}{N}
\left(a_{1}+a_{2}+\frac{k-1}{N}
\left( 1-\left(a_{1}+a_{2}\right) \right) \right) .
\label{equatn21}
\end{equation}
which is non negative for all $k\in I$.
The spectral gap is $1-t_{1}=\frac{1}{N}\left(a_{1}+a_{2}\right)$.

\medskip

When $a_{1}=a_{2}=1$, $t_{k}=1-\frac{k}{N^{2}}\left( 2N+1-k\right) $
and the spectral measure is given by
$\mu_{k}=\frac{2N+1-2k}{2N+1-k}\binom{N}{k}/\binom{2N-k}{N}$.
The expected return time to $0$ is $2^{2N}/\sqrt{\pi N}$
whereas the expected return time to $N/2$ is of order
$\sqrt{\pi N}/2$, much smaller.\newline

\medskip

When $a_{1}+a_{2}=1$, $p(u) =a_{1}$ is
constant and the transition probabilities
become affine linear functions of the
state. Here $\pi(x)=\binom{N}{k}a_{1}^{x}\overline{a}_{1}^{N-x}$,
$\mu_{k}=\binom{N}{k}a_{1}^{k}\overline{a}_{1}^{N-k}$
and $t_{k}=1-\frac{k}{N}$.
When $a_{1}=1/2$, the holding probabilities
are $r_{x}=1/2$ and both $\pi (x)$ and
$\mu_{k}$ are symmetric Binomial$(N,1/2)$ distributed.

\bigskip

\textbf{Cases with positive eigenvalues.} We may look for conditions
on the
mechanism $p$ leading to $r_{x}\geq \frac{1}{2}$ in which case the
BD chain is spectrally positive. Assume $p:[0,1]
\rightarrow ( 0,1)$ is continuous, non--decreasing and
so $0<p(0) \leq p(1) <1$.

\medskip

Then, as can easily be checked when $N$ is even:
$$
r_{x}\geq 1/2 \; \forall \, x\in I \; \Leftrightarrow \;
p(1/2) =1/2 \hbox{ with  } p(0) \leq \frac{1}{2}\leq p(1) .
$$
Indeed, imposing $r_{x}\geq 1/2$ for all $x$
leads to $p(u) \geq 1/2$ if $u\geq 1/2$, $p(u) \leq 1/2$ if
$u\leq 1/2$ and  $p(1/2)= 1/2$. Since $p$ is non--decreasing  
these conditions are equivalent to $p(1/2)=1/2$. 
The reciprocal also holds. When $N$ is odd an analogous condition
can be written. When the mutation mechanism satisfies
$0<a_{1}\leq 1-a_{2}<1$, the condition $p( 1/2) =1/2$,
leads to $a_{1}=a_{2}$. However, it is easy to see that the condition  
$a_{1}=a_{2}$ is not necessary for $P$ to be spectrally positive.

\subsection{ Generalized ultrametric case}
Let us examine another triangular matrix $H$ that is also a 
potential matrix. It belongs to the 
class of generalized ultrametric matrices (see \cite{NaV}, \cite{DNST}), 
a class that contains the ultrametric matrices 
introduced in \cite{MMS}.

\bigskip
 
Let $C$ be a nonempty set
strictly contained in $I=\{0,...,N\}$. Denote $C'=I\setminus C$.
We put $C(x)=C$ when $x\in C$ and $C(x)=C'$ otherwise.
Take $\ac, \bc \ge 0$, and put $\cc(x)=\ac$ if $x\in C$
and $\cc(x)=\bc$ otherwise. Now, define the matrix $H_{\alpha,\beta}$ by
$$
H_{\alpha,\beta}(x,y)=\1(x\le y)+\cc(x) \1(x\le y)\1(C(x)=C(y)) \,,
$$
which is a clear generalization of the Siegmund dual because  
$H_{0,0}=H_S$.
It is straightforward to check that $H_{\alpha,\beta}$ belongs to the 
class of potential 
matrices introduced in Subsection \ref{sub34}, indeed
$H_{\alpha,\beta}=(\Id -R)^{-1}$ with
$$
R(x,y)=\1(x= y)-\frac{1}{1+\cc(x)}\1(x= y)+\frac{1}{1+\cc(x)} \1(x+1=y) 
\,.
$$
As it is easily checked $R$ is an irreducible strictly substochastic 
matrix that looses mass through the state $N$. Then   
$$
{H_{\alpha,\beta}}^{-1}=\Id - R , \hbox{ so } 
{H_{\alpha,\beta}}^{-1}(x,y)=\frac{1}{1+\cc(x)}\1(x= y)-
\frac{1}{1+\cc(x)} \1(x+1=y)\,.
$$
In this case we are able to compute the inverse matrix 
${H_{\alpha,\beta}}^{-1}$, the description
of the inverse of any generalized ultrametric matrices 
can be found in \cite{DMS}.
We point out that $R'$ is substochastic only when $\alpha\ge 
\beta$ and in this case it is an irreducible strictly substochastic 
that looses mass through the state $0$. In the rest of this Subsection, we 
will put $H=H_{\alpha,\beta}$ to avoid overburden notation.

\bigskip

We have,
\begin{eqnarray}
\nonumber
(H \hP')(x,y)&=& \sum\limits_{z\ge x} H(x,z)\hP(y,z)=
\sum\limits_{z\ge x} \hP(y,z)+\cc(x)
\sum\limits_{z\ge x,\, z\in C(x)} \hP(y,z) \,,\\
\nonumber
(PH)(x,y)&=&\sum\limits_{z\le y} P(x,z)H(z,y)=\sum\limits_{z\le y}
P(x,z)+\cc(y) \sum\limits_{z\le y,\, z\in C(y)} P(x,z)\,.
\end{eqnarray}
By permuting $I$ we can always assume that $C$ is an interval, that 
is
$C=\{1,...,k\}$ for some $0\le k<N$, and so $C'=\{k+1,...,N\}$.
With this choice we have that each $x\not\in \{k,N\}$ verifies
$C(x)=C(x+1)$ and so $\cc(x)=\cc(x+1)$.

\bigskip

\noindent $(i)$. Let $x\neq k$. From the above equalities we find
$$
(H\hP')(x,y)-(H\hP')(x+1,y)=(1+\cc(x)) \hP(y,x)\,.
$$
(the case $x=N$ follows from $(H \hP')(N+1,y)=0$). Then the
equality $H \hP'=PH$ implies,
$$
(1+\cc(x))\hP(y,x)= \sum\limits_{z\le y} (P(x,z)-P(x+1,z))
+\cc(y) \sum\limits_{z\le y,\, z\in C(y)} (P(x,z)-P(x+1,z))\,.
$$

\noindent $(i1)$. Let $x\neq k$ and $y\le k$. In this case we
have $\cc(y)=\ac$ and $z\le y$ implies $z\in C(y)$. So, we find
$$
(1+\cc(x)) \hP(y,x)= (1+\ac)\sum\limits_{z\le y} 
(P(x,z)-P(x+1,z))\,.
$$
Then, a necessary and sufficient condition for $\hP(y,x)\ge 0$ is 
that
\begin{equation}
\label{equatn22}
\sum\limits_{z\le y} P(x+1,z) \le \sum\limits_{z\le y} P(x,z) \,,
\end{equation}
and we get
\begin{equation}
\label{equatn23}
\hP(y,x)=\left(\frac{1+\ac}{1+\cc(x)}\right) \sum\limits_{z\le y}
(P(x,z)-P(x+1,z))\,.
\end{equation}

\noindent $(i2)$. Let $x\neq k$ and $y> k$. In this case we have 
that
$\cc(y)=\bc$, and $C(z)= C(y)$ if and only if $z>k$. Then,
$$
(1+\cc(x))\hP(y,x)= \sum\limits_{z\le y} (P(x,z)-P(x+1,z))
+\bc \sum\limits_{k<z\le y} (P(x,z)-P(x+1,z))\,,
$$
and so
\begin{equation}
\label{equatn24}
\hP(y,x)=\left(\!\frac{1}{1+\cc(x)}\!\right) \sum\limits_{z\le k}
(P(x,z)-P(x+1,z))+\left(\!\frac{1+\bc}{1+\cc(x)}\!\right) 
\sum\limits_{k<z\le y} \!\!(P(x,z)-P(x+1,z))\,.
\end{equation}
Then, a necessary and sufficient condition in order that $\hP(y,x)\ge 0$
for $x\neq k$ is that
\begin{equation}
\label{equatn25}
\sum\limits_{z\le k} \! P(x+1,z)\!+\!
(1\!+\!\bc) \sum\limits_{k<z\le y} \! P(x+1,z)
\le \sum\limits_{z\le k} \! P(x,z) \!+\!
(1\!+\!\bc)\sum\limits_{k<z\le y} \! P(x,z)\,.
\end{equation}
We can summarize subcases $(i1)$ and $(i2)$ as
for all $x\neq k$
$$
\hP(y,x)=\left(\frac{1+\cc(y)}{1+\cc(x)}\right) \sum\limits_{z\le y}
(P(x,z)\!-\!P(x+1,z))-\1(y\!>\!k) \frac{\bc}{1+\cc(x)}
\sum\limits_{z\le k} (P(x,z)\!-\!P(x+1,z)).
$$
The necessary and sufficient condition in order that $\hP(y,x)\ge 0$ 
for $x\neq k$ is constituted by (\ref{equatn22}) and (\ref{equatn25}).

\bigskip

\noindent $(ii)$. Assume $x=k$.
Recall that $\cc(k)=\ac$ and $\cc(k+1)=\bc$, so
\begin{eqnarray}
\nonumber
(H \hP')(k,y)&=&\sum\limits_{z\ge k} \hP(y,z)+
\ac \sum\limits_{z\ge k,\, z\in C(k)} \hP(y,z)=
\sum\limits_{z> k} \hP(y,z)+(1+\ac) \hP(y,k)\,.\\
\nonumber
(H\hP')(k+1,y)&=&\sum\limits_{z> k} \hP(y,z)+
\bc \sum\limits_{z> k,\, z\in C(x)} \hP(y,z)=
(1+\bc) \sum\limits_{z> k} \hP(y,z)\,.
\end{eqnarray}
Then, by using $H \hP'=PH$,
\begin{eqnarray*}
\nonumber
(1\!+\!\ac) \hP(y,k)&=&(H \hP')(k,y)\!-\!(H \hP')(k+1,y)\!+\!
\bc \sum\limits_{z>k}\hP(y,z)\\
&=& (PH)(k,y)\!-\left(\frac{1}{1\!+\!\bc}\right)
(PH)(k\!+\!1,y)\,,
\end{eqnarray*}
and so
\begin{eqnarray}
\nonumber
(1\!+\!\ac) \hP(y,k)&=&
\sum\limits_{z\le y}
\left(P(k,z)-\left(\!\frac{1}{1\!+\!\bc}\!\right)P(k+1,z)\right)\\
\label{equatn26}
&{}&+ \cc(y) \!\!\! \sum\limits_{z\le y,\, z\in C(y)} \!
\left(\!P(k,z)\!-\!\left(\!\frac{1}{1\!+\!\bc}\!\right)
P(k\!+\!1, z)\!\right).
\end{eqnarray}
From (\ref{equatn26}) we deduce,
\begin{equation}
\label{equatn27}
y\le k: \;\; \hP(y,k)=
\sum\limits_{z\le y}
\left(P(k,z)-\left(\!\frac{1}{1+\bc}\!\right)P(k+1,z)\right)\,,
\end{equation}
and
\begin{eqnarray}
\nonumber
y> k: \;\; \hP(y,k)&=&\left(\!\frac{1 }{1+\ac}\!\right)
\!\sum\limits_{z\le k}\!\left(P(k,z)\!-
\!\left(\!\frac{1}{1+\bc}\!\right)\!P(k+1,z)\!\right)\\
\label{equatn28}
&{}&+\! \left(\!\frac{1+\bc }{1\!+\!\ac}\!\right)
\!\!\!\sum\limits_{k<z\le y} \!\!\!\left(\!P(k,z)\!-
\!\left(\!\frac{1}{1+\bc}\!\right)\!P(k+1,z)\!\right).
\end{eqnarray}
Hence, the equations (\ref{equatn22}) and (\ref{equatn25}) imply that
$\hP(y,k)\ge 0$ for all $y$, and then they are necessary and 
sufficient for $\hP\ge 0$.

\bigskip

From (\ref{equatn23}) we find,
$$
\forall \, y\le k: \;\;
\sum\limits_{x<k} \hP(y,x)= \sum\limits_{z\le y}
(P(0,z)-P(k,z))\,,\;\;\;
\sum\limits_{x>k} \hP(y,x)= \left(\frac{1+\ac}{1+\bc}\right)
\sum\limits_{z\le y} P(k+1,z)\,.
$$
So, by using (\ref{equatn27}) we get
\begin{equation}
\label{equatn29}
\forall \, y\le k: \;\;\;
\sum\limits_{x\le N} \hP(y,x)= \sum\limits_{z\le y} P(0,z)+
\left(\frac{\ac}{1+\bc}\right)
\sum\limits_{z\le y} P(k+1,z)\,.
\end{equation}
On the other hand, from  (\ref{equatn24}) we obtain,
\begin{eqnarray}
\nonumber
\forall \, y\!> \!k: \;
\sum\limits_{x<k} \!\hP(y,x)&=&
\left(\! \frac{1}{1\!+\!\ac}\!\right) \! \sum\limits_{z\le k}\!
(P(0,z)\!-\!P(k,z))\!+\!\left(\!\frac{1\!+\!\bc}{1\!+\!\ac}\!\right)
 \!\! \sum\limits_{k<z\le y} \!\!(P(0,z)\!-\!P(k,z))
\,,\\
\nonumber
\sum\limits_{x>k} \hP(y,x)&=& \left(\frac{1}{1\!+\!\bc}\right)
\sum\limits_{z\le k} \! P(k\!+\!1,z) \!+\!
\sum\limits_{k<z\le y} \! P(k\!+\!1,z)\,.
\end{eqnarray}
By using (\ref{equatn28}) we get
\begin{eqnarray}
\nonumber
\forall \, y> k: \;\;
\sum\limits_{x\le N}\! \hP(y,x)&=&
\left(\!\frac{1}{1\!+\!\ac}\!\right) \sum\limits_{z\le k}\!
P(0,z)\!+\!\left(\!\frac{1\!+\!\bc}{1\!+\!\ac}\!\right)
\! \sum\limits_{k<z\le y} \!\! P(0,z)
\,,\\
\label{equatn30}
&{}& \; + \!\left(\!\frac{\ac}{(1\!+\!\ac)(1\!+\!\bc)}\!\right)
\!\sum\limits_{z\le k} \! P(k+1,z)\! + \!
 \left(\!\frac{\ac}{1\!+\!\ac}\!\right)
\! \sum\limits_{k<z\le y} \!\! P(k+1,z)\,.
\end{eqnarray}

\begin{proposition}
\label{prop7}
Let $P$ be a stochastic kernel and let $\hP$ be a 
$H_{\alpha,\beta}-$dual of $P$. Then,
a sufficient condition to have $\hP\ge 0$ is the following one:
\begin{eqnarray}
\label{equatn31}
\exists \delta\in (0,1) &{}& \hbox{such that } \forall x\in
\{0,\cdots,N\}: \;\; \sum\limits_{z\le k} P(x,z)= \delta\,;\\
\label{equatn32}
\forall \, y\le k: &{}& \sum\limits_{z\le y} P(x,z)
\hbox{ decreases in } x\in \{1,..,k\}\,; \\
\label{equatn33}
\forall \, y>k:  &{}& \sum\limits_{k<z\le y} P(x,z)
\hbox{ decreases in } x\in \{k+1,..,N\}\,.
\end{eqnarray}
Moreover, under the conditions (\ref{equatn31}), (\ref{equatn32})
and (\ref{equatn33}), $\hP$ is substochastic if and only if
$\delta=\left(\frac{1+\bc}{1+\ac+\bc}\right)$. In this case
$\hP$ is conservative at sites $k$ and $N$.
\end{proposition}

\begin{proof}
The relations (\ref{equatn31}), (\ref{equatn32}) and (\ref{equatn33}), 
are sufficient for $\hP\ge 0$ because they imply the conditions
(\ref{equatn22}) and (\ref{equatn25}).
Now put
$$
L(y)=\sum\limits_{x\le N} \hP(y,x)\,.
$$
From (\ref{equatn29}) we find
that $\{L(y): y\le k\}$ attains its maximum
at $y=k$ and by using (\ref{equatn31}) this maximum becomes
$L(k)=\delta+\left(\frac{\ac \delta}{1+\bc}\right)$. So,
this last quantity must be at most $1$ in order that $\hP$ is
substochastic. On the other hand, from (\ref{equatn30}) it follows that
$\{L(y): y> k\}$ attains its maximum at $y=N$ and that this maximum is
$$
L(N)=
\left(\frac{\delta}{1+\ac}\right)+
\left(\frac{(1+\bc)(1-\delta)}{1+\ac}\right)+
\left(\frac{\ac \delta}{(1+\ac)(1+\bc)}\right)
+\left(\frac{\ac(1-\delta) }{1+\ac}\right).
$$
By straightforward computations it follows that
$$
L(N)=\frac{1}{1+\ac}\left(1+\ac+\beta(1-L(k))\right)\,.
$$
Then, by using  $L(k)\le 1$ we deduce that $L(N)\le 1$
if and only if  $L(k)= 1$, in which case  $L(N)= 1$.
The result is shown. $\Box$
\end{proof}

If the ultrametric dual is seen as a perturbation of the Siegmund 
dual then there is a rigidity result for the BD chains.

\begin{proposition}
\label{prop8}
Let $P$ be the stochastic kernel of an irreducible BD chain on 
$I=\{0,\cdots, N\}$. Assume that there exists a substochastic
kernel $\hP$ that is a $H_{\alpha,\beta}-$dual of $P$,  
$H_{\alpha,\beta} \hP'=P H_{\alpha,\beta}$. 

\medskip

Then we necessarily have $\beta=0$ and the monotone 
property (\ref{equatn14}) is verified.
Moreover, if $k\ge 1$ then $\alpha=\beta=0$ and 
$H_{\alpha,\beta}=H_{0,0}=H_S$ is the Siegmund dual. 

\medskip

If $k=0$ then $\alpha\le (1-p_0)/q_1$. If $\alpha= (1-p_0)/q_1$ 
the kernel $\hP$ is stochastic, and when 
$\alpha< (1-p_0)/q_1$ the kernel $\hP$ is substochastic and it only
looses mass trough $\{0\}$.
\end{proposition}

\begin{proof}
From (\ref{equatn24}) we have
\begin{eqnarray*}
\hP(k\!+\!2,k\!-\!1)\!\!\!\!\!&=&\!\!\!\!\!\left(\!\frac{1}{1+\ac}\!\right)\! 
\sum\limits_{z\le k}\!(P(k\!-\!1,z)\!-\!P(k,z))\!+\! 
\left(\!\frac{1\!+\!\bc}{1\!+\!\ac}\!\right)\!\!
\sum\limits_{k<z\le k+2} \!\!\!\!\!\!(P(k\!-\!1,z)\!-\!P(k,z))\\
\!\!\!\!\!&=&\!\!\!\!\!
\left(\frac{1}{1+\ac}\right)(1\!-\!(1\!-\!P(k,k+1)))-
\left(\frac{1+\bc}{1+\ac}\right)P(k,k+1)\,.
\end{eqnarray*}
So $\hP(k+2,k-1)=-P(k,k+1)(\bc/{1+\ac})$, and we must 
necessary have $\bc=0$.

\bigskip

Since $\bc=0$, from relations (\ref{equatn22}) and (\ref{equatn25}), it 
results 
that the conditions to have $\hP\ge 0$
is that (\ref{equatn14}) is fulfilled, that is
$p_x+q_{x+1}\le 1\;\,$
$\;\forall x\in \{0,\cdots, N-1\}$.
 
\bigskip

On the other hand if $k\ge 1$ we get from (\ref{equatn29}) that
for $y=k$,
$$
\sum\limits_{x\le N} \hP(k,x)= \sum\limits_{z\le k} P(0,z)+
\left(\frac{\ac}{1+\bc}\right)
\sum\limits_{z\le k} P(k+1,z)=1+\left(\frac{\ac}{1+\bc}\right)  
P(k+1,k)\,.
$$
So, we must necessary have $\ac=0$.

In the case $k=0$ from relation (\ref{equatn30}) it results that
$\sum\limits_{y\in I} \hP(x,y)=1$ for all $y>0$. 
The only case we must examine
is (\ref{equatn29}) for $k=0$ and the condition $\sum\limits_{y\in I} 
\hP(0,y)=(1-p-0)+\alpha q_1\le 1$ implies $\alpha\le (1-p_0)/q_1$.
\end{proof}

\section{ Strong Stationary Times}
\label{sec5}
Let $P$ be an irreducible positive recurrent stochastic kernel on
the countable set $I$ and  $X=(X_n: n\ge 0)$ be a Markov chains 
with kernel $P$.
Let $\pi$ be the stationary probability measure of $X$.
We denote by $\pi_0$ the initial distribution of $X$ and in general
$\pi_n$ is the distribution of $X_n$,  
$\pi_n(\cdot)=\PP_{\pi_0}(X_n=\cdot)$. It verifies $\pi'_n=\pi'_0 P^n$.

\medskip

A random time $T$  is called a strong stationary time for $X$, 
if $X_{T}$ has distribution $\pi$ and it is
independent of $T$, see \cite{AD}. The separation 
discrepancy is defined by,
$$
\hbox{sep}\left({\pi }_{n},{\pi }\right) :=
\sup\limits_{y\in I}\left[ 1-\frac{\pi _{n}(y)}{\pi(y)}\right] \,.
$$
It satisfies sep$(\pi_n,\pi)
\geq \left\| \pi_n-\pi \right\| _{TV}$ where
$\left\| \pi_n-\pi \right\| _{TV}=\frac{1}{2}
\sum\limits_{y\in I}\left| \pi_{n}(y) -\pi(y)\right|$ 
is the total variation distance between ${\pi}_{n}$ and ${\pi}$, 
see \cite{AD} and \cite{DF}.
In  Proposition 2.10 in \cite{AD} it was proven that
every stationary time $T$ verifies
\begin{equation}
\text{sep}\left({\pi }_{n},{\pi }\right) \leq
\PP_{\pi_0}\left(T>n\right) \,\;\; n\ge 0\,.
\label{equatn34}
\end{equation}
Based upon this result the strong stationary time $T$
is called sharp when there is equality in (\ref{equatn34}), that is
\begin{equation*}
\text{sep}\left({\pi }_{n},{\pi }\right)=
\PP_{\pi_0}\left(T>n\right)\,\;\; n\ge 0\,.
\end{equation*}
In Proposition 3.2 in \cite{AD} it was shown that a sharp strong stationary 
time always exists.

\bigskip

Let $\tP$ be  a stochastic kernel on the countable set $\tI$ such that 
$\tP$ is a $\Lambda-$intertwining 
of $P$ where $\Lambda$ is a nonsingular stochastic kernel,
so $\tP \Lambda=\Lambda P$. Let $\tX=(\tX_n: n\ge 0)$ be a Markov chain 
with kernel $\tP$. 

\medskip

Recall that when we are in the framework of Theorem \ref{theo1}, 
we have $\tP\Lambda=\Lambda \vP$, so
$\tP$ is a $\Lambda-$intertwining of the reversal kernel $\vP$. 
Hence, when the intertwining is constructed from a dual relation, 
$\vP$ and the reversed chain $\vX$ will play the role of $P$ and   
$X$ in the intertwining relation. 
In the reversible case $\vP=P$ both notations coincide, 
that is $\vP=P$ and we can take $\vX=X$, this
occurs for instance when $P$ is the kernel of an irreducible BD chain.

\bigskip

The initial probability distributions of the chains $X$ and $\tX$        
will be respectively $\pi_{0}$ and ${\widetilde \pi_{0}}$,
that is $X_{0}\overset{d}{\sim}\pi _{0}$ and 
$\tX_{0}\overset{d}{\sim}\widetilde{\pi}_{0}$.
We assume that the initial distributions are linked, this means:
\begin{equation}
\label{eqno0}
\pi'_0={\widetilde \pi}'_0\Lambda.
\end{equation}
When this relation is verified we say that $\pi'_0$ and $\pi_0$ is 
an admissible condition.
Let $\pi_n$ and ${\widetilde \pi }_n$ be the distributions
of $X_{n}$ and $\tX_{n}$. By the intertwining 
relation $\tP^n\Lambda=\Lambda P^n$ for all $n\ge 1$, and the initial 
condition (\ref{eqno0}) we get 
\begin{equation*}
\pi'_{n}= {\widetilde{\pi }}'_{n}\Lambda \;\;\;\, \forall \, n\ge 0\,.
\end{equation*}

\subsection{ The coupling }
Let us recall the coupling done in \cite{DF}
between the intertwining Markov chains. Consider the kernel $\ovP$ 
defined on $I\times \tI$ by:
\begin{equation*}
\ovP\left((x,\widetilde{x}),(y,\widetilde{y})\right)=
\frac{P(x,y) \, \widetilde{P}(\widetilde{x},\widetilde{y})
\, \Lambda (\widetilde{y},y)}{(\Lambda P)(\widetilde{x},y) }
\1\left((\Lambda P)(\widetilde{x},y) >0\right)\,.
\end{equation*}
The kernel $\ovP$ is stochastic.
Let $\ovX=(\ovX_n: n\ge 0)$ be the chain taking values
in $I\times \tI$, evolving with the kernel $\ovP$ and having 
as initial distribution the vector 
$(\pi_0 , {\widetilde \pi}_0)$ where $\pi'_0={\widetilde \pi}'_0\Lambda$.
It can be checked that  $\ovX$ is a coupling of the chains $X$ and $\tX$.
Then, in the sequel we will write by $X$ and $\tX$ 
the components of $\ovX$, so  $\ovX_n=(X_n, \tX_n)$ for all $n\ge 0$. 
In the above construction it can be also checked that, 
\begin{equation}
\label{equatn35}
\Lambda \left( \widetilde{x},x\right)
=\PP\left( X_{n}=x\mid \widetilde{X}_{n}=\widetilde{x}\right)\;\;\,
\forall n\geq 0 \,.
\end{equation}
(For this equality also see \cite{cpy}). In \cite{DF} 
this coupling was characterized as the unique one that verifies 
(\ref{equatn35}) and three other properties on conditional independence. 
These properties imply that the coupling also satisfies,
\begin{equation*}
\Lambda( \widetilde{x}_n,x_n)
=\PP\left( X_{n}=x_n\mid  \widetilde{X}_{0}=\widetilde{x}_0
\cdots \widetilde{X}_{n}=\widetilde{x}_n\right)\;\;\,   
\forall \, n\geq 0\,.
\end{equation*}

In this process the original ergodic Markov chain $X$ governed
by $P$, may be viewed as a random output of the Markov process
$\tX$ governed by $\widetilde{P}=\Lambda P\Lambda ^{-1}$, 
when $\Lambda$ is non singular. 
This is a setup reminiscent of filtering
theory with $\widetilde{X}$ the hidden process and $X$ the
observable. The peculiarity of the intertwining construction is 
that the output $X$ process is itself Markov. 

\bigskip

The following concept was introduced in \cite{DF}.

\medskip

\begin{definition}
\label{def3}
The Markov chain $\tX$ will be called a strong stationary dual of 
the Markov chain $X$, if  $\tX$ has an absorbing state $\tpa$
that verifies
\begin{equation}
\label{equatn36}
\pi(x)=\PP\left( X_{n}=x\mid  \widetilde{X}_{0}=\widetilde{x}_0
\cdots \widetilde{X}_{n-1}=\widetilde{x}_{n-1}, 
\widetilde{X}_{n}=\tpa\right)\;\;\,  \forall x\in I, \, n\geq 0\,,
\end{equation}
and where $\widetilde{x}_0\cdots \widetilde{x}_{n-1}\in \tI$ satisfy
$\PP\left(\widetilde{X}_{0}=\widetilde{x}_0 \cdots 
\widetilde{X}_{n-1}=\widetilde{x}_{n-1}, \widetilde{X}_{n}=\tpa\right)>0$.
\end{definition}

In Theorem 2.4 in \cite{AD} it was shown that
when the condition (\ref{equatn36}) holds then the absorption time 
$\ptT_{\tpa}$ at $\{\tpa\}$ is a strong stationary time for $X$.
Moreover in Remark 2.8 in \cite{DF} it is built a specific dual process
$\tX$ having an absorbing state $\tpa$ and whose absorption time
$\ptT_{\tpa}$ is sharp.

\bigskip

Assume that $\tpa$ is an absorbing state for $\tX$.
From (\ref{equatn6''}) we get $\pi'=\mathbf{e}'_{\tpa} \Lambda$.
When the initial conditions are linked by relation
(\ref{eqno0}) $\pi'_0={\widetilde \pi}'_0\Lambda$,
we get that $\ptT_{\tpa}$ is a strong stationary time for $X$.
Indeed, from $\Lambda(\tpa,x)
=\PP\left( X_{n}=x\mid  \widetilde{X}_{0}=\widetilde{x}_0
\cdots \widetilde{X}_{n}=\tpa\right)$, it follows that
$\pi(x)=\PP\left( X_{n}=x\mid  \widetilde{X}_{0}=\widetilde{x}_0
\cdots \widetilde{X}_{n}=\tpa\right)$ is verified
because condition $\pi'=\mathbf{e}'_{\tpa} \Lambda$ holds.
Observe that for the Siegmund dual and monotone
kernels (that is verifying (\ref{equatn8'})) 
the absorbing state is $\tpa=N$.

\subsection { Choice of the initial conditions.}
Let $\tpa$ be an absorbing state of $\tX$.
From (\ref{eqno0}) the initial conditions of the chains must verify
$\pi'_{0}={\widetilde \pi_{0}}'\Lambda$ to be able to perform the 
duality construction and to get that the absorption time 
${\ptT}_{\tpa}$ is a strong stationary time for $X$.

\bigskip

Assume that $I=\tI$. 
Since $\Lambda$ is a stochastic matrix it has a left probability
eigenvector ${\pi}'_{\Lambda}$
satisfying $\pi'_{\Lambda}=\pi'_{\Lambda}\Lambda $. 
So, we can choose $\tX_{0}\overset{d}{\sim}
X_{0}\overset{d}{\sim } \pi_{\Lambda}$ because
(\ref{eqno0}) is satisfied (we also use $\overset{d}{\sim}$
to mean 'distributed as').
Then, when $\tX$ is initially distributed as $\pi_{\Lambda}$, 
$\ptT_{\tpa}$ is a strong stationary time for the chain $X$
starting from $\pi_{\Lambda}$.

\bigskip

If $\Lambda $ is non irreducible then ${\pi}_{\Lambda}$ could
fail to be strictly positive.
This is the case for the Siegmund kernel. 
In fact, from (\ref{eqn02}) 
it can be checked that $\mathbf{e}_0$ is the unique left eigenvector 
satisfying $\mathbf{e}_0'=\mathbf{e}_0' \Lambda$ and so 
${\pi}_{\Lambda}=\mathbf{e}_0$. Then, for the $\Lambda$-intertwining
given by (\ref{eqn02}) the initial condition 
$\widetilde{X}_{0}\overset{d}{\sim}\delta_{0}$
and $X_{0}\overset{d}{\sim }\delta_{0}$ is admissible.

\bigskip

When ${\widetilde b}\in \tI,\, b\in I$, ${\widetilde b}\neq \tpa$, 
verify $\mathbf{e}'_{\widetilde b}\Lambda =\mathbf{e}'_{b}$ then
$\widetilde{X}_{0}\overset{d}{\sim}\delta_{\widetilde b}$
and $X_{0}\overset{d}{\sim }\delta_{b}$ is an admissible
initial condition (it verifies (\ref{eqno0})). Then
$\ptT_{\tpa}$ starting from $\widetilde b$ is a strong stationary
time for $X$ starting from $b$, and it is strictly positive. In this 
case, both $\widetilde{X}_{0}$ and  $X_{0}$ start at a single point. 
We observe that the condition
$\mathbf{e}_{\widetilde b}'\Lambda=\mathbf{e}_{b}'$ is equivalent 
to the following condition on the dual function: 
$H\mathbf{e}_{\widetilde b}=c\mathbf{e}_{b}$ for some $c>0$. 
Indeed if $H$ verifies this 
condition and since $\Lambda= D_\varphi H' D_\pi$ (see (\ref{equatn5})) 
we obtain $\mathbf{e}_{\widetilde b}'\Lambda=c'\mathbf{e}_{b}'$
with $c'=c\pi(b)/\varphi(\widetilde b)$. Since $\Lambda$ is stochastic we 
get $c'=1$, and so $c=\varphi(\widetilde b)/\pi(b)$. This gives
$H\mathbf{e}_{\widetilde b}=\frac{\varphi(\widetilde b)}{\pi(b)}\mathbf{e}_{b}$,
which is exactly $\mathbf{e}_{\widetilde b}'\Lambda=\mathbf{e}_{b}'$. 

\bigskip

For the Siegmund kernel and $P$ monotone, 
$\Lambda$ is given by (\ref{eqn02}) and 
the equation (\ref{eqno0}) takes the form
$$
\pi_0(x)=\sum_{z=x}^N {\widetilde \pi}_0(z) \frac{\pi(x)}{\pi^c(z)}
\;\;\; \forall x\in I\,.
$$
So we need that $\pi_0(x)/\pi(x)$ decreases with $x\in I$
and in this case 
${\widetilde \pi}_0(x)=
{\pi^c(x)}\left(\pi_0(x)/\pi(x)-\pi_0(x+1)/\pi(x+1)\right)$.
These are, respectively, condition (4.7) and formula (4.10) in \cite{DF}.

\medskip

We recall that every monotone kernel $P$ verifies
condition  $\pi'=\mathbf{e}'_{N} \Lambda$ (see (\ref{equatn6''})). 
The $\Lambda-$intertwining $\tP$ is the one of $\vP$, and in this case 
$\vX$ and $\tX$ denote the Markov chains associated to $\vP$
and $\tP$, respectively. 

\subsection{ Conditions for sharpness }
We now give a proof of the sharpness result alluded
to in Remark $2.39$ of \cite{DF} and in Theorem 2.1 in \cite{F}.

\bigskip

\begin{proposition}
\label{prop9}
Let $X$ be an irreducible positive recurrent Markov chain, $\tX$ be a 
$\Lambda-$intertwining of $X$ having $\tpa$ as an absorbing state.
Assume that there exists $d\in I$ such that 
\begin{equation}
\label{equation37'}
\Lambda \mathbf{e}_{d}=\pi(d)\, \mathbf{e}_{\tpa}.
\end{equation}
Then $\tX$ is a sharp dual to $X$, that is
for $\tX_{0}\overset{d}{\sim }\widetilde{\pi}_{0}$ and
$X_{0}\overset{d}{\sim }\pi_{0}$
with $\pi'_{0}=\widetilde{\pi}'_{0}\Lambda$, we have: 
\begin{equation}
\label{equatn38}
\hbox{sep}(\pi_{n},\pi) =
\PP_{{\widetilde \pi}_0} (\ptT_{\tpa}> n) \; \; \forall \, n\ge 0\,.
\end{equation}
\end{proposition}

\begin{proof}
From condition $\Lambda \mathbf{e}_{d}=\pi(d)\mathbf{e}_{\tpa}$
we get,
\begin{equation}
\label{equatn39}
\pi_n(d) = {\pi}'_{n}\mathbf{e}_{d}=
{\widetilde \pi}'_{n}
\Lambda \mathbf{e}_{d}=\pi(d)\widetilde{\pi}_{n}(\tpa) .
\end{equation}
Since $\pi>0$, the last equalities imply that
\begin{equation}
\label{equatn40}
\pi_n(d)>0 \, \Leftrightarrow \, \widetilde{\pi}_{n}(\tpa)>0\,.
\end{equation}
On the other hand the condition $\pi'=\mathbf{e}'_{\tpa} \Lambda$
means that the $\tpa-$row of $\Lambda$ verifies
$\Lambda(\tpa,\cdot) =\pi'(\cdot)>0$. Then, if for some $n$ we have
$\widetilde{\pi}_{n}(\tpa)>0$, from $\pi'_n={\widetilde \pi}'_{n}\Lambda$
we deduce ${\pi}_{n}>0$. Moreover,
$$
\pi_n(x)=\sum\limits_{{\widetilde x}\in \tI}
{\widetilde \pi}({\widetilde x})\Lambda({\widetilde x}, x)\ge
{\widetilde \pi}(\tpa) \Lambda(\tpa, x)={\widetilde \pi}(\tpa)\pi(x)
$$
Therefore, from (\ref{equatn39}) we get
$$
\min\limits_{x\in I}\frac{\pi_{n}(x)}{\pi(x)}=
\widetilde{\pi}(\tpa)=\frac{\pi_{n}(d)}{\pi(d)}\,.
$$
Then, sep$(\pi_n,\pi)=1-{\widetilde \pi}(\tpa)$. Since $\tpa$ is an 
absorption
state implies  ${\widetilde \pi}_n(\tpa)=
\PP_{{\widetilde \pi}_0}(\ptT_{\tpa}\le n)$, we get the desired
relation
$$
\hbox{sep}(\pi_n,\pi)\!=\!\PP_{{\widetilde \pi}_0}(\ptT_{\tpa}> n)
\;\; \forall \, n\ge n_+, \hbox{ with }
n_{+}\!=\!\inf\{n\ge 0: {\widetilde \pi}_n(\tpa)>0\}\,.
$$
Let us show that the relation (\ref{equatn38}) holds for $n<n_+$.
First remark that in this case ${\widetilde \pi}_n(\tpa)=0$,
which by (\ref{equatn40}) implies $\pi_n(d)=0$. Then
sep$(\pi_n,\pi)=1$ and so the equality
sep$(\pi_n,\pi)=\PP_{{\widetilde \pi}_0}(\ptT_{\tpa}> n)=1$ holds.
We have proven that $\tX$ is a sharp dual to $X$.
\end{proof}

\begin{proposition}
\label{prop9'}
\noindent $(i)$ Assume the hypotheses of Theorem \ref{theo1} are verified
and that $\hP$ is a substochastic kernel having $\widehat a$ as an absorbing 
state in $\hP$. Then, if there exists some $d\in I$ such that
\begin{equation}
\label{equatn40'}
\mathbf{e}'_d H=c \, \mathbf{e}'_{\widehat a} \, \hbox{ for some } c>0\,,
\end{equation}
then $\widehat a$ is an absorbing state for 
$\tX$ and $\tX$ is a sharp dual to $\vX$. That is,
when $\pi'_{0}=\widetilde{\pi}'_{0}\Lambda$ the relation
(\ref{equatn38}) is verified.

\medskip

\noindent $(ii)$ Assume the hypotheses of Theorem \ref{theo1} are verified
and that $\hP$ is a substochastic kernel verifying that there exist
${\widehat a}\in \hI$, $d\in I$ 
such that for some constants $c'>0$, $c>0$ we have   
\begin{equation}
\label{equatn41'}
H\mathbf{e}_{\widehat a}=c' \1 \, \hbox{ and } \,
\mathbf{e}'_d H=c \mathbf{e}'_{\widehat a}
\end{equation}
Then part $(i)$ holds, and $\tX$ is a sharp dual to $\vX$.
\end{proposition}

\begin{proof}
$(i)$ From Theorem \ref{theo1} $(v)$ it follows that $\widehat a$
is an absorbing state for $\tP$. From Proposition \ref{prop9} it
suffices to show that $d$ verifies (\ref{equation37'}): 
$\Lambda \mathbf{e}_{d}=\pi(d)\, \mathbf{e}_{\widehat a}$.
Since the hypothesis is $H(d,y)=c\delta_{y,\widehat a}$ for some $c>0$ and 
for all $y\in \tI$, the Remark \ref{rem3} implies 
$\Lambda(x,d)=c''\delta_{x,\widehat a}$ for some $c''>0$.
Now, from Theorem \ref{theo1} $(v)$ and (\ref{equatn6''}) we have
$\pi(d)=\Lambda(\widehat a,d)$, and we deduce $c''=\pi(d)$. Therefore
$\Lambda(x,d)=\pi(d)\delta_{x,\widehat a}$ which is equivalent to 
(\ref{equation37'}).

\medskip

\noindent $(ii)$ From Proposition \ref{propo1} we get that 
the first relation in (\ref{equatn41'})
guarantees that $\widehat a$ is an absorbing state for $\hP$. 
So, we are under the hypotheses of part $(i)$ and the result follows.
\end{proof}

\begin{corollary}
\label{coro5}
\noindent $(i)$ For a monotone irreducible stochastic 
kernel $P$, the $\Lambda-$intertwining Markov chain 
$\tX$ has $N$ as absorbing state and it is a sharp dual of $\vX$. 
Moreover, both chains $\vX$ and $\tX$ can start at state $0$. 

\medskip

\noindent $(ii)$ For 
a monotone irreducible stochastic BD kernel $P$ we have
that the BD chain $\tX$ is a sharp dual to $X$.
\end{corollary}

\begin{proof} 
For part $(i)$, the properties required for sharpness for the 
Siegmund intertwining of BD chains follow straightforward 
because the $N-$th row of $H_S$ verifies (\ref{equatn40'})
with $d=N$. 
Also the relation (\ref{equatn15}) in Corollary \ref{coro21}
is exactly (\ref{equation37'}). The fact that
the state $0$ is admissible for both $X$ and $\tX$ is
a consequence of $\mathbf{e}_{0}'=\mathbf{e}_{0}'\Lambda$.
In part $(ii)$ the only novelty is that for BD chains $P=\vP$.
\end{proof}

We note that, by definition, for an absorbing point $\widehat a$  
there is a unique state $d$ verifying (\ref{equatn40'}),
as it occurs for the Siegmund kernel. 

\medskip

When  $d$ verifies the property (\ref{equatn40})   
we say that $d$ is a witness state in $X$ that $\tX$ hits $\tpa$.
It reflects the following more general situation.
Assume that $\Lambda$ fulfills $\Lambda(x,y)>0 \Leftrightarrow x\ge y$. 
Then, from $\pi'={\widetilde \pi}'\Lambda$ we get
$$
{\widetilde \pi}_0(x)>0 \, \Rightarrow \pi_0(y)>0 \, \;\;\;
\forall \, y\le x.
$$ 
Then if $N$ is an absorbing state of $\tP$ and  
$P(y,y+1)>0$ and $\tP(y,y+1)>0$ for all $y\in \{0,\cdots,N-1\}$
the equivalence $\pi_n(N)>0\Leftrightarrow {\widetilde \pi}_n(N)>0$
is satisfied, and so $N$ will be a witness state in $X$ that $\tX$ hits 
the state $N$.

\subsection{ Times to absorption}
\label{subsec5.4}
From Proposition \ref{prop9'}, 
for the BD chains the random time $\ptT_{N}$ starting from 
the state $0$
gives information on the speed of convergence
to its invariant measure, of the original BD chain $X$ 
starting from the state $0$. In the sequel we 
denote by $\ptT_{N;0}$ a random variable distributed as
the hitting time $\ptT_N$ when starting from $0$, that is
$\PP(\ptT_{N;0}=n)=\PP_0(\ptT_N=n)$
for $n\ge 1$. We denote its variance by 
$\hbox{Var}(\ptT_{N;0})$. 
 
\bigskip

For BD chains absorbed at $N$, the probability generating
function of $\ptT_{N}$ starting from $0$ is, see \cite{K} and \cite{F},
\begin{equation}
\EE\left(u^{\ptT_{N;0}}\right) =\prod_{k=1}^{N}
\frac{( 1-t_{k}) u}{1-t_{k}u}\text{, }u\in [0,1],
\label{equatn41}
\end{equation}
where $-1<t_{k}<+1$, $k=1,...,N$ are the $N$ distinct eigenvalues of both
$\tP$ and $P$, avoiding $t_{0}=1.$ The formula (\ref{equatn41}) also
reads
\begin{equation*}
\PP\left(\ptT_{N;0}>n\right) =\sum_{l=1}^{N}
\prod_{k\neq l} \frac{1-t_{k}}{t_{l}-t_{k}}t_{l}^{n}\,, \;\; n\geq N-1.
\end{equation*}
Then, $t_{1}^{-n}\PP(\ptT_{N;0}>n) \rightarrow
\prod_{k=2}^{N}\frac{1-t_{k}}{t_{1}-t_{k}}$ 
as $t\to \infty$, and
$\ptT_{N;0}$ has geometric tails with exponent $t_{1}$.
Also, 
\begin{equation*}
\EE(\ptT_{N;0}) = \sum_{k=1}^{N}(1-t_{k})^{-1}\text{ and }
\hbox{Var}(\ptT_{N;0})= \sum_{k=1}^{N}(1-t_{k})^{-2}
-\sum_{k=1}^{N}(1-t_{k})^{-1}\,.
\end{equation*}
Since $t_{1}$ is the dominant eigenvalue
\begin{equation}
\hbox{Var}\left(\ptT_{N;0}\right) \leq
{\EE(\ptT_{N;0})}/{1-t_{1}}.  
\label{equatn42}
\end{equation}
When the eigenvalues $t_{k}$ are non negative, then
$\ptT_{N;0}\overset{d}{\sim}\sum_{k=1}^{N}\tau _{k}$
where the $\tau _{k}$s are
independent and $\tau _{k}\overset{d}{\sim }$ Geometric$(1-t_{k})$,
the geometric distribution with success parameter $1-t_{k}$ on 
$\{1,2,\cdots\}$. Assume that the eigenvalues $t_{k}$ are not all 
positive and put $t_{N}<...<t_{l+1}<0\leq t_{l}<...<t_{1}<t_{0}=1$.
Then (\ref{equatn41}) interprets as:
\begin{equation*}
\ptT_{N;0}-\sum_{k=l+1}^{N}B_{k}\overset{d}{\sim}
\sum_{k=1}^{l}\tau _{k},
\end{equation*}
where $B_{k}\overset{d}{\sim }$ Bernoulli$( 1/\left( 1-t_{k})
\right) ,$ $\tau _{k}\overset{d}{\sim }$ Geometric$( 1-t_{k}) $ and
$\ptT_{N;0}$ are all mutually independent. 
All the previous results in this Subsection \ref{subsec5.4}
can be found in \cite{AD}, \cite{DF} and \cite{F}.

\bigskip

When $t_{k}$ are known explicitly it is possible
to compute $\EE(\ptT_{N;0}) $ and 
Var$(\ptT_{N;0})$. So, we can search for conditions 
under which
\begin{equation*}
\EE(\ptT_{N;0}) \rightarrow \infty \hbox{ and }
{\hbox{Var}(\ptT_{N;0})}/
{\left(\EE(\ptT_{N;0})^2\right)}
\rightarrow 0\text{ as }N\rightarrow \infty .
\end{equation*}
If this is the case, 
$\ptT_{N;0}/{\EE(\ptT_{N;0})}\rightarrow 1$ as $N\to \infty$
in probability, and
$\left\lfloor \EE(\ptT_{N;0}) \right\rfloor$
is a cutoff time for $X$ started at $0$.
In this goal, from (\ref{equatn42}) we get
Var$\left({\ptT_{N;0}}/{\EE(\ptT_{N;0})}\right)
\leq 1/\left((1-t_{1}) \EE(\ptT_{N;0})\right)$.
Then, $(1-t_{1}) \EE(\ptT_{N;0})\rightarrow \infty$ as $N\to \infty$
is a sufficient condition for
Var$\left(\ptT_{N;0}/\EE(\ptT_{N;0})\right)
\rightarrow 0$.
%If this holds, the contribution of
%$\sum_{k=2}^{N}\left( 1-t_{k}\right) ^{-1}$ to $\mu _{N}$ dominates the lead
%term $\left( 1-t_{1}\right) ^{-1}.$
See \cite{DSC} for recent developments and precisions.

\bigskip

\emph{Example:} Consider the Moran model with mutations, and put 
$a:=a_{1}+a_{2}$, $\overline{a}:=1-a$. From (\ref{equatn21}) the 
eigenvalues $t_{k}$ verify: 
$1-t_{k}=\frac{k}{N}\left(a +
\overline{a}\frac{k-1}{N}\right)$. Using the approximation
\begin{equation*}
\EE(\ptT_{N;0}){\sim} N\int_{0}^{1}\frac{dx}{\left( 
x+1/N\right)
\left(a +\overline{a }x\right) }=
\frac{N^{2}}{Na -\overline{a}}\left(\int_{0}^{1}\frac{dx}{x+1/N}-
\overline{a}\int_{0}^{1}\frac{dx}{a +\overline{a}x}\right) ,
\end{equation*}
we get
\begin{equation*}
\EE(\ptT_{N;0})\sim {N\left(\log N+\log a\right)}/a 
\text{ and }\hbox{Var}(\ptT_{N;0}) \sim \left({N}/{a}\right)^{2}
\end{equation*}
showing that Var$\left(\ptT_{N;0}/
\EE(\ptT_{N;0}) \right) \sim (\log N)^{-2}
\rightarrow 0$. 
The expected mixing time is $\EE(\ptT_{N;0})\sim N\log N/a$ 
whereas the spectral gap is $1-t_{1}=a/N$. $\Box$

\bigskip

In general, the values $t_{k}$ are not known. So it would be helpful
to compute differently the mean and the variance of the absorption time 
$\ptT_{N;0}$. This is the goal of our next
paragraph in the BD chain context.

\bigskip

\noindent {\it The mean and the variance of the absorption time}. Let 
us compute $\EE(\ptT_{N;0})$ and $\hbox{Var}(\ptT_{N;0})$ by
the usual methods. We introduce the following sequences of independent 
random variables: 
$$
(S_y: y=0,\cdots, N-1) \hbox{ with distribution } 
\PP(S_y=n)=\PP_y(\ptT_{y+1}=n) \;\;\, \forall \, n \ge 0\,,
$$ 
so $S_y$ is a copy of  
the time spent in hitting $y+1$ when $\tX$ starts from $y$.
We also assume that the sequence $(S_y: y=0,\cdots, N-1)$ 
is independent of the Markov chain $\tX$. Observe that 
$$
\PP(\sum_{y=0}^{N-1} S_y=n)=\PP(\ptT_{N;0}=n)  \;\;\,
\forall \, n \ge 0\,.
$$ 

When the initial condition is $\tX_0=y$, we have the representation
\begin{equation}
\label{equatn42}
S_y\overset{d}{\sim}\1(\tX=y+1)+\1(\tX=y1)
(1+S'_y)+\1(\tX=y-1)(1+S_{y-1}+S''_y)\,,
\end{equation}
where $S'_y$ and $S''_y$ are independent copies of $S_y$, which are 
independent from $\tX$ and from the whole sequence 
$(S_y: y=0,\cdots, N-1)$.
By taking expected values we find the recurrence relation
\begin{equation}
\label{equatn43}
\EE(S_y)=\frac{1}{\widetilde{p}_y}+
\frac{\widetilde{q}_y}{\widetilde{p}_y}\EE(S_{y-1})\,.
\end{equation}
Since $\EE(S_0)=1/{\widetilde{p}_0}$ we get by iteration,
\begin{equation}
\label{equatn44}
\EE(S_y)=\sum_{l=0}^y \frac{1}{\widetilde{p}_y} 
\prod_{r=l+1}^y\frac{\widetilde{q}_r}{\widetilde{p}_r}
\end{equation}
and so the mean of the absorption time at $N$ starting from $0$ is,
$$
\EE(\ptT_{N;0})=\sum_{y=0}^{N-1}
\left(\sum_{l=0}^y \frac{1}{\widetilde{p}_y}
\prod_{r=l+1}^y\frac{\widetilde{q}_r}{\widetilde{p}_r}\right)\,.
$$
Also from (\ref{equatn42}) we obtain
\begin{eqnarray*}
S_y^2&\overset{d}{\sim}&\1(\tX=y+1)+\1(\tX=y)(1+2S'_y+{S'}_y^2)\\
&{}&\, +\1(\tX=y-1)(1+S_{y-1}^2+{S''}_y^2+2S_{y-1}+2S''_y+2S_{y-1}2S''_y)\,.
\end{eqnarray*}
Therefore 
\begin{eqnarray*}
\EE(S_y^2)\!=\!\frac{1}{\widetilde{p}_y}+\frac{2\widetilde{r}_y} 
{\widetilde{p}_y}\EE(S_y)\!+\!
\frac{\widetilde{q}_y}{\widetilde{p}_y} \EE(S_{y-1}^2)
\!+\!\frac{2\widetilde{q}_y}{\widetilde{p}_y}
\left(\!\EE(S_{y-1})\!+\! \EE(S_y)\!+\!\EE(S_{y-1})\EE(S_y)\!\right)\,.
\end{eqnarray*}
From (\ref{equatn43}) we find that
Var$(S_y)=\EE(S_y^2)-\EE(S_y)^2$ verifies
\begin{eqnarray*}
\hbox{Var}(S_y)&=& 
\frac{1}{\widetilde{p}_y}\!+\!\frac{2\widetilde{r}_y} 
{\widetilde{p}_y}\EE(S_y)\!+\!
\frac{\widetilde{q}_y}{\widetilde{p}_y}\EE(S_{y-1}^2\!)
\!+\!\frac{2\widetilde{q}_y}{\widetilde{p}_y}
\left(\EE(S_{y-1}\!)\!+\! \EE(S_y) \!+\!\EE(S_{y-1}\!) \EE(S_y)\right)\\
&{}&\,-\frac{1}{\widetilde{p}_y^2}-
\frac{\widetilde{q}_y^2}{\widetilde{p}_y^2}\EE(S_{y-1})^2
-\frac{2\widetilde{q}_y}{\widetilde{p}_y^2}\EE(S_{y-1})
\,.
\end{eqnarray*}
Therefore
\begin{equation}
\label{equatn45}
\hbox{Var}(S_y)=\frac{\widetilde{q}_y}{\widetilde{p}_y}\hbox{Var}(S_{y-1})
+A_y\,,
\end{equation}
where
\begin{eqnarray*}
A_y&=&
\frac{\widetilde{p}_y-1}{\widetilde{p}_y^2}+\frac{2(1-\widetilde{p}_y)}
{\widetilde{p}_y}\EE(S_y)+
\frac{2\widetilde{q}_y(\widetilde{p}_y-1)}{\widetilde{p}_y^2}\EE(S_{y-1})
+\frac{2\widetilde{q}_y}{\widetilde{p}_y}
\EE(S_{y-1}) \EE(S_y)\\
&{}&\,- \frac{\widetilde{q}_y(\widetilde{q}_y-\widetilde{p}_y)}
{\widetilde{p}_y^2}\EE(S_{y-1})^2
\,.
\end{eqnarray*}
Observe that from (\ref{equatn44}) the coefficient $A_y$ can be 
computed
in terms of the parameters of the BD chain $\tX$.
In particular 
$A_0=\hbox{Var}(S_0)={(1-\widetilde{p}_0)}/{\widetilde{p}_0^2}$.
From the recurrence formula (\ref{equatn45}) and the value for
Var$(S_0)$, we find the explicit expression,
\begin{eqnarray*}
\hbox{Var}(S_y)=\sum_{l=0}^y
A_l\, \prod_{s=l+1}^y \frac{\widetilde{q}_s}{\widetilde{p}_s}\,.
\end{eqnarray*}
Therefore, by using independence, the variance of the 
hitting time of $N$ starting from $0$ is,
\begin{equation}
\label{equatn46}
\hbox{Var}(\ptT_{N;0})=\sum_{y=0}^{N-1}\hbox{Var}(S_y)
=\sum_{y=0}^{N-1}\left(\sum_{l=0}^y
A_l\, \prod_{s=l+1}^y \frac{\widetilde{q}_s}{\widetilde{p}_s}\right)\,,
\end{equation}
which can be explicitly computed simply 
in terms of the transition parameters of the BD chain
$\tX$.

\medskip

\begin{remark} 
Even if the expressions of the mean and the variance in 
(\ref{equatn44})  
and (\ref{equatn46}) do not require the knowledge of the spectrum, they
are difficult to handle in terms of the parameters, so  
in general we are not able to use them to describe the behavior 
of the mean and the variance when $N$ is large. 
\end{remark}

\section{ The Hypergeometric dual }
For $I=\{0,\cdots,N\}$ let us suggest other potentially interesting examples 
of nonsingular duality kernels $H$ for which there exists a column
of $H$ which is constant so that Proposition \ref{propo1} can be
applied. For these examples, $H^{-1}$ is known explicitly
which turns out to be useful to decide whether for a given
irreducible stochastic kernel the $H-$dual defines a
substochastic matrix. If this occurs, the problem of interpreting the
intertwining chain given by Theorem \ref{theo1},
remains a challenging problem for each specific case.

The Vandermonde dual and the hypergeometric duals that were first 
introduced in \cite{MM} in
the context of neutral population genetics. 
In this context and also in nonneutral situations, 
the hypergeometric kernel plays a central role.

\medskip

\noindent $\cdot \;$  \textbf{Vandermonde.}
$H( x,y)=(x/N)^{y}$. In this case the column $0-$th is constant.

\medskip

\noindent  $\cdot \;$ \textbf{Hypergeometric.} $H(x,y)
=\binom{N-x}{y}/\binom{N}{y}$. In this case $H=H'$, 
$H$ is upper-left triangular,
and (\ref{equatn41'}) is verified with $\widehat a=0$
and $d=N$,
\begin{equation}
\label{equatn42'}
H\mathbf{e}_{0}=\1 \, \hbox{ and } \,
\mathbf{e}'_N H=\mathbf{e}'_{0}\,.
\end{equation}
Let us comment on this choice of $H$.

\medskip

\noindent When $P$ is given by the reversible Moran
model with completely monotone non-neutrality bias mechanism,
the $H-$dual kernel $\hP$ can be interpreted
in terms of a multi-sex backward process akin to the coalescent.
As shown in \cite{HM}, for the Moran model with bias $p$ 
satisfying $p(0)\in (0,1)$ we have:
$\hP \1(0)=1$ and $0<\hP \1(x)=1-\frac{x}{N}p(0) <1$ for all $x\neq 0$. 
From $p(0) \neq 0$, all the states but $a=0$ of
$\hP$ are mass-defective. The intertwining matrix $\tP$ is 
the transition kernel of a skip-free to the 
left BD chain that can easily be obtained from \cite{HM}, and
$0$ is the unique absorbing state for $\tX$. 
The relation (\ref{equatn42'}) fulfills the hypotheses of  
Proposition \ref{prop9'} with $d=N$,
then in the above Moran model the sharpness property
is satisfied.

\medskip

On the other hand the link matrix $\Lambda$ is upper-left triangular,  
stochastic and irreducible. Then, there exists a probability 
vector $\pi_\Lambda$ that verifies $\pi'_\Lambda \Lambda=\pi'_\Lambda$, 
so $\pi_0={\widetilde \pi}_0= \pi_\Lambda$ is an admissible 
initial condition for $X$ and $\tX$. Also, from
$\mathbf{e}'_N \Lambda=\mathbf{e}'_0$ we get that another admissible 
initial condition is $\pi_0=\delta_0$ and ${\widetilde \pi}_0=\delta_N$.
We can summarize this discussion in the following result.

\begin{corollary}
\label{coro3'}
Let $X$ be the Moran chain with transition matrix $P$ fulfilling the above
monotonicity conditions on $p$. Then, the construction of 
the intertwining kernel $\tP$
in Theorem \ref{theo1} starting from the hypergeometric dual $H$ can
be done and the Markov chain $\tX$ is well-defined.
The absorbing state of $\tX$ is $0$,
the process $\tX$ is a sharp dual of $X$ and $\tX$
can be started at $N$ while $X$ starts at $0$.
\end{corollary}

Hence, the time $\ptT_{0;N}$ that $\tX$ reaches $0$ when it starts from 
$N$, is the stochastically smallest
time at which $X_{\ptT_{0;N}}\overset{d}{\sim } \pi$
given $X_{0}=0$ and $\widetilde{X}_{0}=N$. 
We point out that the time $\ptT_{0;N}$ that $\tX$ reaches $0$ when  
it starts from $N$,
is distributed like the time $\ptT_{N;0}$
to reach $N$ starting from $0$ of the Siegmund intertwining BD chain to 
the Moran model, namely like (\ref{equatn41}). This is in accordance with 
Theorem $1.2$ of \cite{F2}, stating that for a skip-free to the right 
Markov chain
absorbed at $N$, the law of the time it takes to hit $N$ starting from $0$
is given by (\ref{equatn41}). This result can be transferred to our
skip-free to the left BD chain case, while exchanging the boundaries
$\{0,N\}$. 

\smallskip

\noindent Let us finally consider the Wright-Fisher transition matrix 
$P$ given by
\begin{equation*}
P\left( x,y\right) =\binom{N}{y}p\left( \frac{x}{N}\right) ^{y}\left(
1-p\left( \frac{x}{N}\right) \right) ^{N-y}.
\end{equation*}
whose bias $p(u)$ is again a completely monotone function,
satisfying $p(0)>0$. This process is not reversible, nor is it in the BD class.
However, using the hypergeometric duality kernel
it was shown in \cite{H} that the $H-$dual $\widehat{P}$
to $P$ in (\ref{equatn2}) defines a substochastic matrix.
From Theorem \ref{theo1} we conclude
that the corresponding $\widetilde{P}$ is $\Lambda -$linked to
$P$. $\Box$.

\section*{Acknowledgements}
The authors acknowledge the partial support given by the
BASAL-CMM project. S. Mart\'{\i}nez thanks
Guggenheim Fellowship and the hospitality of {Laboratoire de
Physique Th\'eorique et Mod\'elisation at Universit\'e de
Cergy-Pontoise,

\end{document}